\documentclass{emsprocart}
\usepackage[all]{xy}
\newcommand{\ZZ}{\mathbb{Z}}

\newcommand{\PP}{\textbf{P}}

\contact[cerulli.math@googlemail.com]{Giovanni Cerulli Irelli, Dipartimento di Matematica ``G. Castelnuovo'', Sapienza--Universit\'a di Roma, Piazzale Aldo Moro 5, 00185, Roma, ITALY}

\newtheorem{theorem}{Theorem}[section]
\newtheorem{corollary}[theorem]{Corollary}
\newtheorem{lem}[theorem]{Lemma}
\newtheorem{proposition}[theorem]{Proposition}

\newtheorem{example}{Example}[section]

\theoremstyle{definition}
\newtheorem{definition}[theorem]{Definition}
\newtheorem{remark}[theorem]{Remark}

\newcommand{\Hom}{\operatorname{Hom}}
\newcommand{\Ker}{\operatorname{Ker}}
\newcommand{\Ext}{\operatorname{Ext}}
\newcommand{\Leq}{\leq_{\textrm{deg}}}

\title[Quiver Grassmannians  of Dynkin type]{Geometry of Quiver Grassmannians of Dynkin type with applications to cluster algebras}
\author[Giovanni Cerulli Irelli] {Giovanni Cerulli Irelli\thanks{This work was partially financed by the
italian FIRB project ``Perspectives in Lie Theory'' RBFR12RA9W and partially by DFG SPP-1388. 
}}

\begin{document}

\begin{abstract}
The paper includes a new proof of the fact that quiver Grassmannians associated with rigid representations of Dynkin quivers do not have cohomology in odd degrees. Moreover, it is shown that  they do not have torsion in homology. A new proof of the Caldero--Chapoton formula is provided. As a consequence a new proof of the positivity of cluster monomials in the acyclic clusters associated with Dynkin quivers is obtained. The methods used here are based on joint works with Markus Reineke and Evgeny Feigin. 
\end{abstract}

\begin{classification}
Primary 16G70; Secondary 14N05.
\end{classification}

\begin{keywords}
Quiver Grassmannians, Dynkin quivers, Cluster Algebras.
\end{keywords}

\maketitle

\section*{Introduction}
Quiver Grassmannians are the natural generalization in the world of quivers of the usual Grassmannians of linear subspaces: given a representation $M$ of a quiver $Q$ and a dimension vector $\mathbf{e}$, the quiver Grassmannians $Gr_\mathbf{e}(M)$ parametrizes the subrepresentations of $M$ of dimension vector $\mathbf{e}$. The name was suggested by Zelevinsky, and  it is nowadays commonly used. We say that a quiver Grassmannian is of Dynkin type, if it is associated with a complex representation of a Dynkin quiver. 

My first aim in writing this paper was to survey  recents results obtained in collaboration with Evgeny Feigin and Markus Reineke, concerning the geometry of quiver Grassmannians of Dynkin type. While writing it, I was looking for a more conceptual proof of the Caldero--Chapoton formula and I found a surprising result which then changed my original plan. 

The Caldero--Chapoton formula is a  formula which expresses the cluster variables of a cluster algebra associated with a Dynkin quiver in terms of Euler characteristic of some quiver Grassmannians of Dynkin type. After this formula appeared,  it was hoped that a better knowledge of the geometry of quiver Grassmannians would improve our knowledge on the corresponding cluster algebra. This paper provides a new evidence of this general philosophy which was the heart of my project ``categorification of positivity in cluster algebras'' financed by the  DFG priority program SPP-1388. 

Let me briefly explain the main results of the paper.  Given a dimension vector $\mathbf{d}\in
\ZZ^{Q_0}_{\geq0}$, we denote by $R_\mathbf{d}$ the affine space parametrizing $Q$--representations of dimension vector $\mathbf{d}$ (see section~\ref{Sec:1} for details). The group $G_\mathbf{d}=\prod_{i\in Q_0} GL_{d_i}(K)$ acts on $R_\mathbf{d}$ by base change (K denotes the field of complex numbers) and $G_\mathbf{d}$--orbits correspond to isoclasses of $Q$--representations. Given $M\in R_\mathbf{d}$ and a dimension vector  $\mathbf{e}$, we denote by $Gr_\mathbf{e}(M)$ the projective variety consisting of sub--representations of $M$ of dimension vector $\mathbf{e}$ (see section~\ref{Sec:QG} for details). It can easily be proved that
$$
\textrm{dim }Gr_\mathbf{e}(M)\geq\langle\mathbf{e},\mathbf{d-e}\rangle
$$
where $\langle-,-\rangle$ denotes the Euler form of Q. If equality holds, then we say that $Gr_\mathbf{e}(M)$ has minimal dimension. Since $Q$ is Dynkin, for  every dimension vector $\mathbf{d}$ there exists a representation, that we denote by $\tilde{M}_\mathbf{d}$, whose orbit is dense in $R_\mathbf{d}$. This module is uniquely (up to isomorphisms) determined in $R_\mathbf{d}$ by the condition $\textrm{Ext}^1(\tilde{M}_\mathbf{d}, \tilde{M}_\mathbf{d})=0$ which means that it is rigid.  It can be proved that if $Gr_\mathbf{e}(\tilde{M}_\mathbf{d})$ is non--empty then it is smooth and irreducible of minimal dimension (see Theorem~\ref{Thm:RigQuivGras}). Non--emptiness of  $Gr_\mathbf{e}(\tilde{M}_\mathbf{d})$ is equivalent to $\textrm{Ext}^1_Q(\tilde{M}_\mathbf{e},\tilde{M}_\mathbf{d-e})=0$ (see Corollary~\ref{Cor:Criterion NonEmpty}). The surprising result mentioned above is the following:
\begin{theorem}\label{Thm:NewIntro}
Let $\mathbf{e},\mathbf{d}$ be dimension vectors for Q, such that $\mathbf{d-e}$ is again a dimension vector. Then all non--empty quiver Grassmannians of the form  $Gr_\mathbf{e}(M)$ (where $M\in R_\mathbf{d}$) which are smooth and of minimal dimension are diffeomorphic to each other. In this case, they are all diffeomorphic to $Gr_\mathbf{e}(\tilde{M}_\mathbf{d})$ and hence they are irreducible and they share the same Poincar\'e polynomial and  Euler characteristic.  
\end{theorem}

From this result, the formula of Caldero--Chapoton follows immediately. Let me explain this point. Let 
$$
\xymatrix{
0\ar[r]&\tau M\ar[r]&E\ar[r]&M\ar[r]&0}
$$
be an almost split sequence for Q. As recalled in Section~\ref{Sec:CCFormula}, the CC--formula is a direct consequence of the  following equality
\begin{equation}\label{CCIntro}
\chi(Gr_\mathbf{e}(E))=\chi(Gr_\mathbf{e}(M\oplus\tau M))\qquad\forall\mathbf{e}\neq\mathbf{dim }M
\end{equation}
where $\chi$ denotes the Euler--Poincar\'e characteristic. 
It can be proved (see Proposition~\ref{Prop:QGMTauM}) that $Gr_\mathbf{e}(M\oplus\tau M)$ is smooth of minimal dimension, and hence by  theorem~\ref{Thm:NewIntro}, it is diffeomorphic to $Gr_\mathbf{e}(E)$, since $E$ is rigid. In particular, \eqref{CCIntro} holds. Moreover the Poincar\'e polynomial $P_{Gr_\mathbf{e}(M\oplus\tau M)}(q)$ of $Gr_\mathbf{e}(M\oplus\tau M)$ equals the Poincar\'e polynomial $P_{Gr_\mathbf{e}(E)}(q)$ of $Gr_\mathbf{e}(E)$ for $\mathbf{e}\neq\mathbf{dim }M$. This fact, implies the new formula:  
\begin{align}\label{Eq:FormulaQuantum}
P_{Gr_\mathbf{e}(E)}(q)&=&\sum_{\mathbf{f+g=e}}q^{2\langle\mathbf{f},\mathbf{dim }\tau M-\mathbf{g}\rangle}P_{Gr_\mathbf{f}(M)}(q)P_{Gr_\mathbf{g}(\tau M)}(q)\\\nonumber
&=&\sum_{\mathbf{f+g=e}}q^{2\langle\mathbf{g},\mathbf{dim }M-\mathbf{f}\rangle}P_{Gr_\mathbf{f}(M)}(q)P_{Gr_\mathbf{g}(\tau M)}(q)
\end{align}
(the second equality follows by Poincar\'e duality). In \cite{Qin} it is shown that the Poincar\'e polynomials of certain quiver Grassmannians  are the coefficients of the quantum F--polynomials of quantum cluster variables. The proof is obtained by reduction to the non--quantum case. I wonder if formula~\eqref{Eq:FormulaQuantum} can provide a direct proof of this.

Another interesting consequence of Theorem~\ref{Thm:NewIntro} is the following result
\begin{theorem}\label{thm:noOddIntro}
Let E be a rigid representation of a Dynkin quiver. Then every quiver Grassmannian $Gr_\mathbf{e}(E)$ associated with E has no odd cohomology. In particular $\chi(Gr_\mathbf{e}(E))\geq0$. 
\end{theorem}
The fact that quiver Grassmannians of Dynkin type have positive Euler characteristic was first proved by Caldero and Keller \cite[Theorem~3]{CK1}. The fact that quiver Grassmannians associated with rigid representations of acyclic quivers have no odd cohomology was proved by Qin \cite{Qin} and by Nakajima \cite{Naka}  using the Decomposition Theorem and Fourier--Sato--Deligne transform. After the work of Nakajima, other proofs of the positivity conjecture appeared in the literature \cite{Efimov}, \cite{Davison} as consequence of results on quantum cluster algebras. The proof given here for Dynkin quivers, which I believe can be extended to a more general setting, is based on ideas developed in joint papers with E. Feigin and M. Reineke \cite{CFR,CFR2, CFR3} and uses only well--known results of algebraic/differential geometry (namely the Ehresmann's localization theorem \cite{Ehr} and the Bialynicky--Birula theorem \cite[Theorem~4.1]{BB}). 

The following result concerns (singular) homology groups of quiver Grassmannians which are smooth of minimal dimension. 

\begin{theorem}
Let $X$ be a smooth  quiver Grassmannian of minimal dimension (of Dynkin type). Then $H_i(X)$ is zero if i is odd and it has no torsion if i is even.
\end{theorem}

The paper is fairly self--contained and does not assume any particular knowledge from the reader. Precise references are provided, and most of the results are fully proved.  It is organized as follows: in Section~\ref{Sec:1} we recall the basics on representation theory of quivers and (classical) Auslander--Reiten theory. In Section~\ref{Sec:QG} some properties of quiver Grassmannians are recalled. In Section~\ref{Sec:Degeneration} we discuss a theorem of Bongartz, concerning degenerations of quiver representations and its consequences for quiver Grassmannians. Section~\ref{Sec:Cluster} contains the main results of the paper. 
Finally, Section~\ref{Sec:GVect} contains applications to cluster algebras and the complete proof of the Caldero--Chapoton formula.

\section{Basics on representation theory of Dynkin quivers}\label{Sec:1}
In this section we recall some basic facts on representations of quivers. Standard references for this are \cite{R}, \cite{ARS}, \cite{ASS}.
A (finite) quiver $Q=(Q_0,Q_1,s,t)$ is an ordered quadruple in which $Q_0$ denotes a finite set of vertices (whose cardinality is always denoted with the letter $n$), $Q_1$ is a finite set of edges, and $s$ and $t$ are two functions $s,t: Q_1\rightarrow Q_0$ which provide an orientation of the edges. For an oriented edge  $\alpha$ we write $\alpha:s(\alpha)\rightarrow t(\alpha)$. A  quiver $Q$ is called Dynkin if the underlying graph $\Delta=(Q_0,Q_1)$ is a (possibly union of) simply--laced Dynkin diagram of type A, D or E shown in figure~\ref{Fig:Dynkin}.
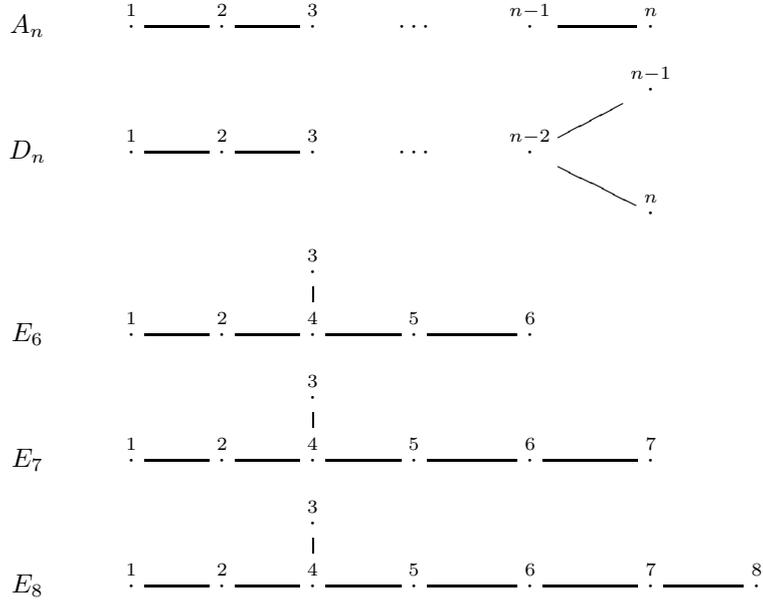
\begin{figure}
$$
\xymatrix@R=5pt{
A_n&\stackrel{1}{\cdot}\ar@{-}[r]&\stackrel{2}{\cdot}\ar@{-}[r]&\stackrel{3}{\cdot}&\cdots&\stackrel{n-1}{\cdot}\ar@{-}[r]&\stackrel{n}{\cdot}\\
&&&&&&\stackrel{n-1}{\cdot}\\
D_n&\stackrel{1}{\cdot}\ar@{-}[r]&\stackrel{2}{\cdot}\ar@{-}[r]&\stackrel{3}{\cdot}&\cdots&\stackrel{n-2}{\cdot}\ar@{-}[ur]\ar@{-}[dr]&\\
&&&&&&\stackrel{n}{\cdot}\\
&&&\stackrel{3}{\cdot}&&&\\
E_6&\stackrel{1}{\cdot}\ar@{-}[r]&\stackrel{2}{\cdot}\ar@{-}[r]&\stackrel{4}{\cdot}\ar@{-}[r]\ar@{-}[u]&\stackrel{5}{\cdot}\ar@{-}[r]&\stackrel{6}{\cdot}&\\
&&&\stackrel{3}{\cdot}&&&\\
E_7&\stackrel{1}{\cdot}\ar@{-}[r]&\stackrel{2}{\cdot}\ar@{-}[r]&\stackrel{4}{\cdot}\ar@{-}[r]\ar@{-}[u]&\stackrel{5}{\cdot}\ar@{-}[r]&\stackrel{6}{\cdot}\ar@{-}[r]&\stackrel{7}{\cdot}\\
&&&\stackrel{3}{\cdot}&&&\\
E_8&\stackrel{1}{\cdot}\ar@{-}[r]&\stackrel{2}{\cdot}\ar@{-}[r]&\stackrel{4}{\cdot}\ar@{-}[r]\ar@{-}[u]&\stackrel{5}{\cdot}\ar@{-}[r]&\stackrel{6}{\cdot}\ar@{-}[r]&\stackrel{7}{\cdot}\ar@{-}[r]&\stackrel{8}{\cdot}
}
$$
\caption{Simply--laced Dynkin diagrams}\label{Fig:Dynkin}
\end{figure}

Let $Q$ be a Dynkin quiver.  A representation of $Q$ is a pair of tuples $V=((V_i)_{i\in Q_0}, (V_\alpha)_{\alpha\in Q_1})$ where $V_i$ is a vector space over a field K and $V_\alpha:V_{s(\alpha)}\rightarrow V_{t(\alpha)}$  is a linear map. In this paper we only deal with complex representations, and hence K denotes always the field of complex numbers. 
A $Q$--morphism $\psi:V\rightarrow W$ from two $Q$--representations is a collections $(\psi_i:V_i\rightarrow W_i)_{i\in Q_0}$ of linear maps such that the following square 
$$
\xymatrix{
V_{s(\alpha)}\ar^{V_{\alpha}}[r]\ar_{\psi_{s(\alpha)}}[d]&V_{t(\alpha)}\ar^{\psi_{t(\alpha)}}[d]\\
W_{s(\alpha)}\ar^{W_{\alpha}}[r]&W_{t(\alpha)}
}
$$
commutes for every arrow $\alpha$ of $Q$. The set of $Q$--morphisms between two $Q$--representations $V$ and $W$ is a vector space that we denote by $\Hom_Q(V,W)$. The category $\textrm{Rep}_K(Q)$ of $Q$--representations is hence an abelian category (with the obvious notions of kernel and cokernel). 

To a quiver $Q$ is associated its (complex) path-algebra $A=KQ$, which is the algebra formed by concatenation of arrows. The category $\textrm{Rep}_K(Q)$ is equivalent to the category $A$--mod of $KQ$--modules. 

The set $\{e_i\}_{i\in Q_0}$ of paths of length zero form a complete set of pairwise orthogonal idempotents of $A$. As a consequence, the projective indecomposable (left) $A$--modules are $\{P_i:=Ae_i\}_{i\in Q_0}$. The path algebra $A=KQ$, viewed as a left $A$--module, decomposes as $A=\bigoplus_{i\in Q_0} P_i$.   As  $Q$--representation, $P_i$ is described as follows: the vector space at vertex $k$ has a basis given by paths from vertex $i$ to vertex $k$, and the arrows act by ``concatenation''. Since $Q$ is the orientation of a tree, every projective $P_i$ is thin, which means that the vector space $(P_i)_k$ associated to every vertex $k$ is at most one--dimensional. Moreover, it is also clear that  $P_i$ has only one maximal sub-representation which hence coincides with its radical and the quotient is the simple $S_i$:
\begin{equation}\label{Eq:ProjPresS}
\xymatrix{
0\ar[r]& \textrm{rad}(P_i)=\bigoplus_{i\rightarrow j} P_j\ar[r] &P_i\ar[r] &S_i\ar[r] &0.
}
\end{equation} 
We notice that \eqref{Eq:ProjPresS} is the minimal projective presentation of $S_i$. 

Dually, the injective indecomposable (left) $A$--modules are the indecomposable direct summands of $DA$ (viewed as left $A$--module), where $D$ is the standard K--duality. They are denoted by $\{I_k\}_{k\in Q_0}$.  As $Q$--representation, $I_k$ has at vertex $j$ a vector space with basis consisting of all the paths of $Q$ starting in $j$ and ending in $k$, and the arrows act by ``concatenation''.  From this, we see that $I_k$ has simple socle $S_k$ and the quotient is injective again: 
\begin{equation}\label{Eq:InjPresS}
\xymatrix{
0\ar[r]& S_k\ar[r] &I_k\ar[r] &\bigoplus_{j\rightarrow k} I_j\ar[r]& 0.
}
\end{equation} 
The short exact sequence \eqref{Eq:InjPresS} is the minimal injective resolution of $S_k$. The opposite algebra $A^{op}$, which is the path algebra of the opposite quiver, viewed as a (left) $A$--module is given by $A^{op}=\bigoplus_{j\in Q_0}I_j$. 

We hence see that subrepresentations of projectives are projectives, and hence the category $\textrm{Rep}_K(Q)$ is hereditary, in the sense that every module has projective and injective dimension at most 1. In other words, every module $M$ admits a minimal projective resolution
of the form
\begin{equation}\label{Eq:MinProjRes}
\xymatrix{
0\ar[r]&P_1^M\ar^{\iota_M}[r]&P_0^M\ar^{\pi_M}[r]&M\ar[r]&0
}
\end{equation}
and a minimal injective resolution: 
\begin{equation}\label{Eq:MinInjRes}
\xymatrix{
0\ar[r]&M\ar[r]&I_0^M\ar[r]&I_1^M\ar[r]&0.
}
\end{equation}
It is not hard to describe the indecomposable direct summands of $P_0^M$ and $P_1^M$:
\begin{equation}\label{Eq:ProjReso}
\begin{array}{cc}
P_0^M=\bigoplus_{i\in Q_0}P_i^{[M,S_i]}&P_0^M=\bigoplus_{i\in Q_0}P_i^{[M,S_i]^1}
\end{array}
\end{equation}
where the following shorthand is used and will be used throughout the text
$$
\begin{array}{cc}
[M,N]:=\textrm{dim }\Hom_Q(M,N)&[M,N]^1:=\textrm{dim }\Ext_Q^1(M,N)
\end{array}
$$
for any representations $M$ and $N$. Similarly, 
\begin{equation}\label{Eq:injeReso}
\begin{array}{cc}
I_0^M=\bigoplus_{j\in Q_0}I_j^{[S_j,M]}&I_1^M=\bigoplus_{j\in Q_0}I_j^{[S_j,M]^1}.
\end{array}
\end{equation}

For a $Q$--representation M, the collection  $(\dim M_i)_{i\in Q_0}\in\ZZ^{Q_0}_{\geq0}$ of non--negative integers is called the \emph{dimension vector} of $M$, and it is denoted in bold by $\mathbf{dim}\, M$. Once the dimension vector is fixed, a $Q$--representation is uniquely determined by linear maps: this leads us to the variety of $Q$--representations. Let $\mathbf{d}=(d_i)_{i\in Q_0}\in\ZZ^{Q_0}_{\geq0}$ be a dimension vector. The vector space 
$$
R_\mathbf{d}:=\bigoplus_{i\in Q_0} \Hom_K(K^{d_i},K^{d_i})
$$
is the \emph{variety of $Q$--representations of dimension vector $\mathbf{d}$}. The group
$$
G_\mathbf{d}:=\prod_{i\in Q_0}\textrm{GL}_{d_i}(K)
$$
acts on $R_\mathbf{d}$ by base change:
$
(g_i)_\alpha\cdot (V_\alpha)_\alpha:=(g_{t(\alpha)}V_\alpha g_{s(\alpha)}^{-1})_\alpha
$
and $G_\mathbf{d}$--orbits are in bijection with isoclasses of $Q$--representations. 

Given another dimension vector $\mathbf{e}\in\ZZ^{Q_0}_{\geq0}$ we consider the vector space of (``degree zero'') $K$--morphisms
$$
\Hom(\mathbf{e},\mathbf{d})=\bigoplus_{i\in Q_0} \Hom_K(K^{e_i},K^{d_i})
$$ 
and the vector space of (``degree one'') $K$--morphisms
$$
\Hom(\mathbf{e},\mathbf{d}[1])=\bigoplus_{\alpha \in Q_1} \Hom_K(K^{e_{s(\alpha)}},K^{d_{t(\alpha)}}).
$$ 
Given  $N\in R_\mathbf{e}$ and $M\in R_\mathbf{d}$ we consider the map
$$
\Phi^M_N:\Hom(\mathbf{e},\mathbf{d})\rightarrow \Hom(\mathbf{e},\mathbf{d}[1]):\;(f_i)_{i\in Q_0}\mapsto (M_\alpha\circ f_{s_{\alpha}}-f_{t(\alpha)}\circ N_\alpha)_{\alpha\in Q_1}
$$
This is a linear map between finite dimensional vector spaces and one can show quite easily (see e.g. \cite{R},  \cite{ASS}): 
$$
\begin{array}{cc}
\Ker \Phi_N^M=\Hom_Q(N,M),&\textrm{CoKer}\, \Phi_N^M\simeq \Ext^1_Q(N,M).
\end{array}
$$
From these formulas we immediately get:
\begin{equation}\label{Eq:Euler}
\dim \Hom_Q(N,M)- \dim \Ext^1_Q(N,M)=\dim \Hom(\mathbf{e},\mathbf{d})-\dim\Hom(\mathbf{e},\mathbf{d}[1]).
\end{equation}
We notice: $\dim \Hom(\mathbf{e},\mathbf{d})=\sum_{i\in Q_0}\!\!e_id_i$ and $\dim\Hom(\mathbf{e},\mathbf{d}[1])=\sum_{\alpha\in Q_1}\!\!e_{s(\alpha)}d_{t(\alpha)}$. In particular, if $\mathbf{d}$ is a dimension vector we get
\begin{equation}\label{Eq:DimRd}
\dim \Hom(\mathbf{d},\mathbf{d}[1])=\dim R_\mathbf{d}.
\end{equation}
Given two arbitrary integer vectors $\mathbf{e},\mathbf{d}\in \ZZ^{Q_0}$ it is hence natural to define the \emph{Euler form of Q} as a bilinear form
$\langle-,-\rangle_Q:\ZZ^{Q_0}\times\ZZ^{Q_0}\rightarrow \ZZ$
 given by
$$
\langle\mathbf{e},\mathbf{d}\rangle:=\sum_{i\in Q_0}e_id_i-\sum_{\alpha\in Q_1}e_{s(\alpha)}d_{t(\alpha)}.
$$
From \eqref{Eq:Euler} above, we immediately get
\begin{equation}\label{Eq:EulForm}
\dim \Hom_Q(N,M)- \dim \Ext^1_Q(N,M)=\langle\mathbf{dim}\,N,\mathbf{dim\,M}\rangle.
\end{equation}
Formula~\eqref{Eq:EulForm} is called the homological interpretation of the Euler form. 

The category $\textrm{Rep}_K(Q)$ is Krull--Schmidt, in the sense that every finite dimensional $Q$--representation can be decomposed in an essentially unique way as direct sum of its  indecomposable direct summands. It is known that a $KQ$--module $M$ is indecomposable if and only if $\textrm{End}_Q(M)\simeq K$ (see e.g. \cite[Corollary VII.5.14]{ASS}  or \cite[Corollary 4.2, Example 4.2]{Ralf}). A famous theorem of P.~Gabriel \cite{Gabriel} (see also \cite{BGP} for a different proof and \cite[Section~VII.5]{ASS} for a survey) states that a quiver $Q$ admits only a finite number of indecomposable representations if and only if $Q$ is a Dynkin quiver.  As a consequence, we get that if $Q$ is Dynkin, then $R_\mathbf{d}$ consists of finitely many $G_\mathbf{d}$--orbits, and hence, since such orbits are connected and locally closed, there is one orbit which is dense. The corresponding representation is called \emph{the generic representation} of dimension vector 
$\mathbf{d}$ and we denote it by $\tilde{M}_\mathbf{d}$. The stabilizer of a point  $M\in R_\mathbf{d}$ is  
$$
\textrm{dim Stab}_{G_\mathbf{d}}(M)=\textrm{Aut}_Q(M)
$$
where $\textrm{Aut}_Q(M)$ denotes the open subvariety of $\Hom_Q(M,M)$ consisting of invertible morphisms. In particular, $\textrm{dim }\textrm{Aut}_Q(M)=\textrm{dim }\Hom_Q(M,M)$. In view of  \eqref{Eq:DimRd} and \eqref{Eq:EulForm}, we hence get
$$
\textrm{codim}_{R_\mathbf{d}}\,(G_\mathbf{d}\cdot M)=\textrm{dim}\,R_\mathbf{d}-\textrm{dim Stab}_{G_\mathbf{d}}(M)=\textrm{dim}\,\Ext^1_Q(M,M).
$$
We conclude that the orbit of $M$ is dense in $R_\mathbf{d}$ if and only if $\Ext^1_Q(M,M)=0$. A representation $M$ such that $\Ext^1_Q(M,M)=0$ is called \emph{rigid}. We thus reformulate the above remark by: for Dynkin quivers a representation is generic if and only if it is rigid. 

The second part of Gabriel's theorem recalled above, states that if $Q$  is a Dynkin quiver, then  the dimension vector restricts to a bijection between the indecomposable $Q$--representations  and the positive roots of the root system associated with the underlying Dynkin graph of Q. Such dimension vector $\mathbf{d}$ satisfies the equation $\langle\mathbf{d},\mathbf{d}\rangle=1$. In particular, for such a dimension vector, the generic representation in $R_\mathbf{d}$ is indecomposable. 

\subsubsection{Almost split sequences}
We conclude this section by recalling the fundamental notions of \emph{irreducible morphism} and of \emph{Auslander--Reiten quiver} of a quiver $Q$ (see e.g \cite{ARS}, \cite{ASS}). A morphism $f:M\rightarrow N$ between two $Q$--representations is called \emph{irreducible} if $f$ is neither mono--spli, nor epi--split (i.e. it does not admit neither a left nor a right inverse) and whenever there is a factorization $f=f_2\circ f_1$, then either $f_1$ is  epi--split  or $f_2$ is  epi--split . It can be shown that if  $f$ is an irreducible monomorphism, then $M$ is indecomposable and, dually, if $f$ is an irreducible epimorphism then $N$ is indecomposable. A short exact sequence 
$$
\xymatrix{
\delta:&0\ar[r]&N\ar^f[r]&E\ar^g[r]&M\ar[r]&0
}
$$ 
is called \emph{almost split} if $f$ and $g$ are irreducible (in particular both $N$ and $M$ are indecomposable) \cite[Proposition~V.5.9]{ARS}. Almost split sequences are characterized by the following property, which will be used later: the short exact sequence $\delta$ is almost split if and only if it is non--split, both $N$ and $M$ are indecomposable and for any morphism $h:X\rightarrow M$ which is not a split--epimorphism, there exists $t:X\rightarrow E$ such that $h=g\circ t$. In particular, if $\delta$ is an almost split sequence, and $M$ is not a direct summand of $X$, then $[X,E]=[X,N\oplus M]$. 
Dually, it can be shown that $\delta$ is almost split if and only if it is non--split, both $N$ and $M$ are indecomposable and for any morphism $h:N\rightarrow X$ which is not a split--monomorphism, there exists $t:E\rightarrow X$ such that $h=t\circ f$. A fundamental result of Auslander and Reiten \cite[Theorem~V.1.15]{ARS} states that for every indecomposable $M$ which is not projective, there exists an almost split sequence $\delta$ as above (ending in $M$), which is unique up to isomorphism \cite[Theorem~V.1.16]{ARS}. One can show that almost split sequences are \emph{rigid}, in the sense that they are uniquely determined (up to isomorphisms) by the three modules $N$, $E$ and $M$ \cite[Proposition~V.2.3]{ARS}. 

In our situation, which is the case of an hereditary algebra of finite representation type, the (indecomposable) module $N$ is obtained in the following way from the (indecomposable non--projective) module $M$: we apply the contra-variant functor $\Hom_Q(-,A): Rep(Q)\rightarrow Rep(Q^{op})$ to the minimal projective resolution \eqref{Eq:MinProjRes} of $M$ and, since $\Hom_Q(M,A)=0$, we get
$$
\xymatrix@C=10pt{
0\ar[r]& \Hom_Q(P_0^M, A)\ar[r]&\Hom_Q(P_1^M,A)\ar[r]&\Ext^1_Q(M,A)\ar[r]&0
}
$$
Now we apply the standard K--duality $D:=\Hom_K(-,K):Rep(Q^{op})\rightarrow Rep(Q)$ and get
\begin{equation}\label{Eq:MinInjResTau1}
\xymatrix@C=10pt{
0\ar[r]& D\Ext^1_Q(M,A)\ar[r]&D\Hom_Q(P_1^M,A)\ar[r]&D\Hom_Q(P_0^M, A)\ar[r]&0.
}
\end{equation}
The composition of endofunctors $D\circ \Hom(-,A)$ of  $\textrm{Rep}_K(Q)$ induces an equivalence $\nu:\textrm{Proj}(Q)\rightarrow \textrm{Inj}(Q)$ between the full subcategory $\textrm{Proj}(Q)$ of projective modules and the full subcategory $\textrm{Inj}(Q)$ of injective modules. The functor $\nu$ is called the \emph{Nakayama} functor and it is characterized by the following property 
\begin{equation}\label{Eq:PropNu}
\nu(P_j)=I_j
\end{equation} 
for all $j\in Q_0$. The short exact sequence \eqref{Eq:MinInjResTau1} is nothing but the minimal injective resolution of the $A$--module $D\Ext^1_Q(M,A)$. We claim that $N\simeq D\Ext^1_Q(M,A)$. To convince ourselves that this is true, we notice that, since $P_0^M$ is projective, there is a non--degenerate bilinear map
$$
\chi: \Hom_Q(P_0^M,M)\times \Hom_Q(M,\nu(P_0^M))\rightarrow K
$$
(to see this, consider the linear map $\omega_M:D\Hom_Q(P_0^M,M)\rightarrow \Hom_Q(M,\nu(P_0^M))$ given by $\omega_M(\psi)(n)(f):=\psi(p\mapsto f(p)n)$. This map is an isomorphism \cite[Proposition~II.4.4(b)]{ARS}). Let $h\in  \Hom_Q(M,\nu(P_0^M))$ such that $\chi(\pi_M, h)\neq 0$. Then the pull--back of \eqref{Eq:MinInjResTau1} by $h$ is an almost split sequence ending in $M$ and hence it is isomorphic to $\delta$. This shows that $N\simeq D\Ext^1_Q(M,A)$. It is customary to denote $\tau M=N=D\Ext^1_Q(M,A)$ and the almost split sequence ending in $M$ becomes
\begin{equation}\label{Eq:AlmostSplit}
\xymatrix{
\delta:&0\ar[r]&\tau M\ar^f[r]&E\ar^g[r]&M\ar[r]&0.
}
\end{equation}

To summarize, we have seen that given the minimal projective resolution \eqref{Eq:MinProjRes} of $M$ then the minimal injective resolution of $\tau M$ is:
\begin{equation}\label{Eq:MinInjResTau}
\xymatrix@C=10pt{
0\ar[r]& \tau M\ar[r]&\nu(P_1^M)\ar[r]&\nu(P_0^M)\ar[r]&0.
}
\end{equation}
This fact will be used in section~\ref{Sec:GVect}. 

To conclude this sub--section we briefly recall the construction of the Auslander--Reiten quiver of a Dynkin quiver Q. This is a quiver, denoted with $\Gamma_Q$, whose vertices are isoclasses of indecomposable $Q$--representations, and there is an arrow $[M]\rightarrow [N]$ if there exists an irreducible morphism $f:M\rightarrow N$ (see e.g. \cite[Section~IV.4]{ASS}). If $Q$ is a connected Dynkin quiver, its AR--quiver is connected and it is explicitely described as follows. Let $\hat{Q}$ be the infinite quiver obtained by ``repeating'' the opposite quiver $Q^{op}$: its vertices are pairs $(i,k)$ with $i\in Q_0$ and $k\in \ZZ$; there is an arrow $(i;k)\rightarrow (j;\ell)$ if either 1) $k=\ell$ and there is an arrow $j\rightarrow i$ in $Q$ or 2) if $j=k+1$ and there is an arrow $i\rightarrow j$ in Q.  Let $\tau:\hat{Q}_0\rightarrow \hat{Q}_0$ defined by $\tau(i;k):=(i;k-1)$. The pair $(\hat{Q},\tau)$ is called a translation quiver. Every vertex of $\hat{Q}$ lies in a unique $\tau$--orbit. The AR--quiver of Q, is obtained from $\hat{Q}$ as follows: consider the set of vertices $(i;0)$ in the zero copy of $Q^{op}$. Identify each such vertex $(i;0)$ with $P_i$. Let $k(i)$ be the unique index such that $\tau^{-k(i)}P_i$ is injective.  Then the AR--quiver is the full subquiver of $\hat{Q}$ containing all vertices $\{(i;k)|\,0\leq k\leq k(i)\}$. Given a vertex $(i;k)$ of $\Gamma_Q$ we denote by $M(i;k)$ the corresponding indecomposable module. The almost split sequences have the form
\begin{equation}\label{Eq:ARQuiverVertices}
\xymatrix@C=15pt{
0\ar[r]&M(i;k)\ar[r]&\displaystyle{
\bigoplus_{j\rightarrow i\in Q_1} M(j;k)\oplus\bigoplus_{i\rightarrow j\in Q_1} M(j;k+1)}\ar[r]&M(i;k+1)\ar[r]&0
}
\end{equation}
I recommend the book \cite{Ralf} for more details about the construction of Auslander-Reiten quivers of Dynkin quivers. 
\section{Quiver Grassmannians}\label{Sec:QG}
Let $Q$ be a finite quiver  with $n$ vertices and let $A=KQ$ be the associated (complex) path algebra. Given a dimension vector $\mathbf{d}$, an  $A$--module $M\in R_\mathbf{d}$ and another dimension vector $\mathbf{e}$ such that $\mathbf{d-e}\in\ZZ^{Q_0}_{\geq0}$, in this section we define the projective variety $Gr_\mathbf{e}(M)$ whose points parametrize submodules of $M$ of dimension vector $\mathbf{e}$.  We need to ask ourselves ``what is a submodule?''. This question has two answers: first of all, a submodule is a collection $(N_i)_{i\in Q_0}$ of vector subspaces $N_i\subseteq M_i$ such that $M_\alpha(N_i)\subseteq N_j$ for every arrow $\alpha:i\rightarrow j$ of Q. On the other hand, a submodule $N\subset M$ can be tought of as an $A$--module $N$ endowed with an injective $A$--morphism $\iota: N\rightarrow M$.  The two answers provide two different realizations of $Gr_\mathbf{e}(M)$.

\subsection{First realization: universal quiver Grassmannians}
Let $\mathbf{d}$ and $\mathbf{e}$ be two dimension vector for $Q$ such that $e_i\leq d_i$ for all $i\in Q_0$. Let us consider the product of usual Grassmannians of  vector spaces over the field K of complex numbers: $Gr_\mathbf{e}(\mathbf{d}):=\prod_{i\in Q_0}Gr_{e_i}(K^{d_i})$. Given $M\in R_\mathbf{d}(Q)$ and a point $N\in Gr_\mathbf{e}(\mathbf{d})$, the condition that $N$ defines a sub-representation of $M$ is $M_\alpha(N_{s_{\alpha}})\subseteq N_{t_{\alpha}}$. We hence consider the incidence variety inside $Gr_\mathbf{e}(\mathbf{d})\times R_\mathbf{d}$ given by:
\begin{equation}\label{Eq:DefUnivQG}
Gr_{\mathbf{e}}^{Q}(\mathbf{d}):=\{(N,M)\in Gr_\mathbf{e}(\mathbf{d})\times R_\mathbf{d}|\, M_\alpha(N_{s_{\alpha}})\subseteq N_{t_{\alpha}},\,\forall \alpha\in Q_1\}.
\end{equation}
The variety $Gr_{\mathbf{e}}^{Q}(\mathbf{d})$ is called the \emph{universal quiver Grassmannian} associated with $\mathbf{e}$, $\mathbf{d}$ and $Q$. The two projections $p_1:Gr_\mathbf{e}(\mathbf{d})\times R_\mathbf{d}\rightarrow Gr_\mathbf{e}(\mathbf{d})$ and $p_2:Gr_\mathbf{e}(\mathbf{d})\times R_\mathbf{d}\rightarrow R_\mathbf{d}$ induce two maps 
$$
\xymatrix{
&Gr_\mathbf{e}^Q(\mathbf{d})\ar_{p_\mathbf{e}}@{->}[dl]\ar^{p_\mathbf{d}}@{->}[dr]&\\
Gr_\mathbf{e}(\mathbf{d})&&R_\mathbf{d}
}
$$
The group $G_\mathbf{d}$ acts diagonally on $Gr_{\mathbf{e}}^{Q}(\mathbf{d})$ and the two maps $p_\mathbf{e}$ and $p_\mathbf{d}$ are $G_\mathbf{d}$--equivariant. Since $Gr_\mathbf{e}(\mathbf{d})$ is a projective variety, the map $p_2$ is proper; moreover $Gr_\mathbf{e}^Q(\mathbf{d})$ is closed in $Gr_\mathbf{e}(\mathbf{d})\times R_\mathbf{d}$ and the closed embedding $Gr_\mathbf{e}^Q(\mathbf{d})\rightarrow Gr_\mathbf{e}(\mathbf{d})\times R_\mathbf{d}$ is proper. It follows that the map $p_\mathbf{d}$ is proper, being the composition of two proper maps.  Its image is the \emph{closed} subset of $R_\mathbf{d}$ consisting of those points $M\in R_\mathbf{d}$ which admit a sub-representation of dimension vector $\mathbf{e}$. The \emph{quiver Grassmannian} $Gr_\mathbf{e}(M)$ associated with a point $M\in R_\mathbf{d}$ is defined as the fiber of $p_\mathbf{d}$ over $M$. 

As shown in \cite[section~2.2]{CFR}, the map $p_\mathbf{e}$ realizes $Gr_\mathbf{e}^Q(\mathbf{d})$ as the total space of an homogeneous vector bundle over $Gr_\mathbf{e}(\mathbf{d})$ of rank $\sum_{\alpha\in Q_1}d_{s(\alpha)}d_{t(\alpha)}+e_{s(\alpha)}e_{t(\alpha)}-e_{s(\alpha)}d_{t(\alpha)}$. In particular, $Gr_\mathbf{e}^Q(\mathbf{d})$ is smooth and irreducible of dimension
$$
\textrm{dim }Gr_\mathbf{e}^Q(\mathbf{d})=\langle\mathbf{e},\mathbf{d}-\mathbf{e}\rangle+\textrm{dim }R_\mathbf{d}.
$$
By upper--semicontinuity of the fiber dimension, we see that for any point $M$ in the image of $p_\mathbf{d}$ we have
\begin{equation}\label{Eq:IneqDimQuivGrass}
\textrm{dim }Gr_\mathbf{e}(M)\geq \textrm{dim }Gr_\mathbf{e}^Q(\mathbf{d})-\textrm{dim }Im(p_\mathbf{d})\geq \langle\mathbf{e},\mathbf{d}-\mathbf{e}\rangle.
\end{equation}
Moreover, since $p_\mathbf{d}$ is $G_\mathbf{d}$--equivariant, the image of $p_\mathbf{d}$ contains a dense orbit of a point whose fiber has dimension precisely $\textrm{dim }Gr_\mathbf{e}^Q(\mathbf{d})-\textrm{dim }Im(p_\mathbf{d})$. If such a dense orbit is the orbit of the rigid representation $\tilde{M}_\mathbf{d}$ of dimension vector $\mathbf{d}$ (see Corollary~\ref{Cor:Criterion NonEmpty}  for a criterion for this), then $p_\mathbf{d}$ is surjective and hence
$$
\textrm{dim }Gr_\mathbf{e}(\tilde{M}_\mathbf{d})=\langle\mathbf{e},\mathbf{d}-\mathbf{e}\rangle.
$$
(For more properties of $Gr_\mathbf{e}(\tilde{M}_\mathbf{d})$ see Theorem~\ref{Thm:RigQuivGras} below.) 

Let $D:Rep(Q)\rightarrow Rep(Q^{op})$ be the standard duality which associates to a $Q$--representation $M$ its linear dual $DM$. There is an isomorphism of projective varieties
\begin{equation}\label{Eq:DualityGrass}
\zeta:Gr_\mathbf{e}(M)\rightarrow Gr_{\mathbf{d}-\mathbf{e}}(DM):\, L\mapsto \textrm{Ann}_M(L):=\{\varphi\in DM|\, \varphi(L)=0\}
\end{equation}
where $\mathbf{d}:=\textbf{dim }M$ and $\mathbf{e}\in\ZZ_{\geq 0}^{Q_0}$ is any dimension vector. 

If $Q=Q'\cup Q''$ is a disjoint union of two sub-quivers, then any $Q$--representation $M$ is a direct sum $M=M'\oplus M''$ of a representation M' of $Q'$ and of a representation $M''$ of $Q''$. Any quiver Grassmannian $Gr_\mathbf{e}(M)$ is a product of the form:
\begin{equation}\label{Eq:QuivGrassProduct}
Gr_\mathbf{e}(M)=Gr_{\mathbf{e}'}(M')\times Gr_{\mathbf{e}''}(M'').
\end{equation}
for some $\mathbf{e}'\in \ZZ_{\geq0}^{Q_0'}$ and $\mathbf{e}''\in \ZZ_{\geq0}^{Q_0''}$.
\subsection{Second realization: quiver Grassmannians as geometric quotients and stratification}\label{Sec:Stratification}
Following Caldero and Reineke \cite{CR}, one  can realize quiver Grassmannians as geometric quotients.
Recall the two vector spaces $\Hom(\mathbf{e},\mathbf{d})$ and $\Hom(\mathbf{e},\mathbf{d}[1])$ of section~\ref{Sec:1} and the linear map $\Phi_L^M:\Hom(\mathbf{e},\mathbf{d})\rightarrow \Hom(\mathbf{e},\mathbf{d}[1])$ associated with $L\in R_\mathbf{e}(Q)$ and $M\in R_\mathbf{d}(Q)$. Let us assume that $e_i\leq d_i$ for all $i\in Q_0$. Given $M\in R_\mathbf{d}(Q)$ the algebraic map
$$
\Phi^M: R_\mathbf{e}\times \Hom(\mathbf{e},\mathbf{d})\rightarrow \Hom(\mathbf{e},\mathbf{d}[1]): (L,f)\mapsto \Phi^M_L(f)
$$
is used to define the following closed subvariety of $R_\mathbf{e}\times \Hom(\mathbf{e},\mathbf{d})$:
$$
\Hom(\mathbf{e},M):=\{(L,f)\in  R_\mathbf{e}\times \Hom(\mathbf{e},\mathbf{d})|\, \Phi^M_L(f)=0\}.
$$
Inside $\Hom(\mathbf{e},\mathbf{d})$ there is the open (and dense) subvariety $\Hom^0(\mathbf{e},\mathbf{d})$ consisting of collections of injective linear maps; the induced open subvariety  $\Hom^0(\mathbf{e},M):=\Hom(\mathbf{e},M)\cap \left(R_\mathbf{e}\times \Hom^0(\mathbf{e},\mathbf{d})\right)$ is of particular importance for us. Indeed the map 
$$
\phi: \Hom^0(\mathbf{e},M)\rightarrow Gr_\mathbf{e}(M):\, (L,f)\mapsto f(L)
$$
is surjective and each fiber of $\phi$ is a free orbit for the algebraic group $G_\mathbf{e}=\prod_{i\in Q_0}GL(e_i)$ (see \cite[Lemma~2]{CR}). This implies that the quiver Grassmannian $Gr_\mathbf{e}(M)$ is a geometric quotient:
\begin{equation}\label{Eq:QuotQuivGrass}
Gr_\mathbf{e}(M)\simeq \Hom^0(\mathbf{e},M)/G_\mathbf{e}.
\end{equation}
With this formulation, a point $p$ of $Gr_\mathbf{e}(M)$ is represented (up to the $G_\mathbf{e}$--action) by a pair $(L,\iota)$ where $L\in R_\mathbf{e}(Q)$ and $\iota:L\rightarrow M$ is an injective homomorphism of $Q$--representations; in this case we use the notation $p=[(L,\iota)]$. As shown by Caldero and Reineke, formula \eqref{Eq:QuotQuivGrass} implies the following description of the (scheme-theoretic) tangent space $T_{p}(Gr_\mathbf{e}(M))$ at a point $p$ of the quiver Grassmannian.
\begin{theorem}
Given $M\in R_\mathbf{d}(Q)$ and a point $p=[(L,\iota)]\in Gr_\mathbf{e}(M)$, there is an isomorphism of vector spaces
$$
T_{p}(Gr_\mathbf{e}(M))\simeq \Hom_Q(L,M/\iota(L))
$$ 
where $T_{p}(Gr_\mathbf{e}(M))$ denotes the (scheme--theoretic) tangent space 
at $p$ of $Gr_\mathbf{e}(M)$. 
\end{theorem}
\begin{remark}
The tangent space formula only holds at level of schemes. The usual example in this sense is given by considering a regular  (indecomposable) representation $R_2$ of the Kronecker quiver of quasi--lenght 2 whose dimension vector is $(2,2)$. The quiver Grassmannian $Gr_{(1,1)}(R_2)$ is a point, but the tangent space has dimension one. 
\end{remark}
Formula \eqref{Eq:QuotQuivGrass} allows to define a stratification of $Gr_\mathbf{e}(M)$ as follows  (see \cite[Section~2.3]{CFR} for more details): let $p$ be the projection from $\Hom^0_Q(\mathbf{e},M)$ to $R_\mathbf{e}$; its fiber over a point $N\in R_\mathbf{e}$ is the space of injective linear maps $\Hom^0_Q(N,M)$. For each isoclass $[N]$ in $R_\mathbf{e}$ we can consider the subset $\mathcal{S}_{[N]}$ of $Gr_\mathbf{e}(M)$ corresponding under the previous isomorphism to $p^{-1}(G_\mathbf{e}\cdot N)/G_\mathbf{e}$. In \cite[Lemma~2.4]{CFR} it is shown that $\mathcal{S}_{[N]}$ is a locally closed subset of dimension
$$
\textrm{dim }\mathcal{S}_{[N]}=[N,M]-[N,N].
$$
In particular, a quiver Grassmannian  $Gr_\mathbf{e}(M)$ admits a finite (since $Q$ is Dynkin) stratification
$$
Gr_\mathbf{e}(M)=\coprod_{[N]}\mathcal{S}_{[N]}.
$$
The irreducible components of $Gr_\mathbf{e}(M)$ are hence closure of some strata which we called the \emph{generic sub--representation types} of $Gr_\mathbf{e}(M)$ (see \cite{CFR3}).

\section{Degeneration of $Q$--representations: Bongartz's theorem and applications to quiver Grassmannians}\label{Sec:Degeneration}
Given $M,N\in R_\mathbf{d}$,  $M$ is said to \emph{degenerate} to $N$ and in this case it is customary to write $M\leq_{\textrm{deg}} N$, if the closure of the orbit of $M$ contains $N$: 
$$
M\Leq N\;\;\;\stackrel{def}{\Longleftrightarrow}\;\;\; \overline{G_\mathbf{d}\cdot M}\supseteq G_\mathbf{d}\cdot N.
$$
For arbitrary finite--dimensional algebras, it is a hard problem to control such a notion. On the other hand, for algebras of finite representation type (i.e. admitting a finite number of indecomposable modules) the following very useful characterization holds: 
\begin{equation}\label{Eq:DegZwara}
\begin{array}{ccccc}
M\Leq N&\Longleftrightarrow&[X,M]\leq [X,N]&\Longleftrightarrow&[M,X]\leq [N,X].\\
&&\forall\;X\in\textrm{Rep}(Q)&&\forall\;X\in\textrm{Rep}(Q)
\end{array}
\end{equation}
For Dynkin quivers this result was obtained by Bongartz \cite{B} (but many other people should be mentioned here: e.g. Riedtmann \cite{Ried}, Abeasis-Del Fra \cite{AdF1, AdF2, AdF3}). The surprising generalization to  any algebra of finite representation type was obtained by  Zwara \cite{Z} (the second equivalence follows from Auslander--Reiten theory \cite[section~2.2]{Z}, \cite{AR85}). 

In the analysis of the geometry of quiver Grassmannians developed in collaboration with Reineke and Feigin, the following result of Bongartz played a prominent r\^ole (it is stated below for Dynkin quivers but it holds in full generality): in order to formulate it we need to recall the notion of a generic quotient from Bongartz's paper \cite[Section~2.4]{B}. Suppose that  $U\in R_\mathbf{e}$ and $M\in R_\mathbf{d}$ are given, and also that there exists a monomorphism $\iota:U\rightarrow M$; in particular $\mathbf{d-e}\in\ZZ^{Q_0}_{\geq0}$ is a dimension vector. The set of all possible quotients of $M$ by $U$ is an  irreducible constructible  subset of $R_{\mathbf{d}-\mathbf{e}}$ which is $G_{\mathbf{d}-\mathbf{e}}$--invariant. In particular, since $Q$ is Dynkin, it is the closure of a $G_{\mathbf{d}-\mathbf{e}}$--orbit of a point $S$ called the \emph{generic quotient} of $M$ by $U$. 

\begin{theorem}\label{Thm:Bongartz}(\cite[Theorem~2.4]{B})
Let $M,N\in R_\mathbf{d}$ such that $M\Leq N$. Let $U$ be a representation such that $[U,M]=[U,N]$ then the following holds: 
\begin{enumerate}
\item if $U$ embeds into $N$,  it embeds into $M$ too;
\item in this case every quotient of $N$ by $U$ is  a degeneration of the generic quotient of $M$ by U.
\end{enumerate}
\end{theorem}

We immediately get an interesting corollary which says that the generic sub-representation of a generic representation is generic and its generic quotient is generic. This was noticed also in the paper of Schofield \cite{S}, but the proof that I give here relies entirely on Bongartz's theorem.  Here is the precise statement:
\begin{corollary}\label{Cor:EmbeddingGeneric}
Let $\tilde{M}_\mathbf{d}$ be a rigid representation of dimension vector $\mathbf{d}$. Let $N\subseteq \tilde{M}_\mathbf{d}$ be a sub-representation of dimension vector $\mathbf{e}$. Then the rigid representation $\tilde{M}_\mathbf{e}$ of dimension vector $\mathbf{e}$ embeds into $\tilde{M}_\mathbf{d}$ with generic quotient $\tilde{M}_\mathbf{d-e}$. In particular, there is a short exact sequence
\begin{equation}\label{Eq:GenericSequence}
\xymatrix{
0\ar[r]&\tilde{M}_\mathbf{e}\ar[r]&\tilde{M}_\mathbf{d}\ar[r]&\tilde{M}_\mathbf{d-e}\ar[r]&0}
\end{equation}
\end{corollary}
\begin{proof}
For simplicity of notation, we put $M:=\tilde{M}_\mathbf{d}$ and $L:=\tilde{M}_\mathbf{e}$.  If $\mathbf{e}$ is either zero or $\mathbf{d}$, there is nothing to prove. Thus, let us assume that $0\subsetneq N\subsetneq M$ is a proper sub-representation of $M$. Then the quotient of $M$ by the image of the embedding $N\subseteq M$ is a non--zero representation of $Q$ of dimension vector $\mathbf{d}-\mathbf{e}\neq\mathbf{0}$ that we denote by the symbol $M/N$ (this notation is misleading since it is not sensitive  to the particular embedding $N\subseteq M$ but it is commonly used). We notice that $[N,M/N]^1=0$: indeed $[N,M/N]^1\leq [M,M]^1=0$ (see \cite[proof of Corollary~3]{CR}).  The representation $R:=L\oplus M/N$ is a representation of dimension vector $\mathbf{d}$ and hence $M\Leq R$. Since $L\Leq N$, in view of \eqref{Eq:DegZwara} , we get $[L,M/N]^1\leq [N,M/N]^1=0$ and hence $[L,R]^1=0$.  In particular, $[L,R]=\langle\mathbf{e},\mathbf{d}\rangle$. 
In view of \eqref{Eq:DegZwara}, we also have $[L,M]\leq [L,R]=\langle\mathbf{e},\mathbf{d}\rangle\leq [L,M]$. In conclusion, $[L,M]=[L,R]=\langle\mathbf{e},\mathbf{d}\rangle$. Since $L$ embeds into $R$ by construction, the first  part of Theorem~\ref{Thm:Bongartz} guarantees that $L$ embeds into $M$ too. The second part of Theorem~\ref{Thm:Bongartz} implies that the generic quotient of $M$ by $L$ degenerates to $\tilde{M}_{\mathbf{d-e}}$. In particular the generic quotient of $M$ by $L$ is $\tilde{M}_{\mathbf{d-e}}$ itself, proving \eqref{Eq:GenericSequence}.
\end{proof}

In corollary~\ref{Cor:EmbeddingGeneric}, the non--emptiness of  the quiver Grassmannian $Gr_\mathbf{e}(\tilde{M}_\mathbf{d})$ was assumed. The next result is a criterion to decide when this is the case. 

\begin{corollary}\label{Cor:Criterion NonEmpty}
Let $\mathbf{d}, \mathbf{e}\in\ZZ^{Q_0}_{\geq0}$ be two dimension vectors such that $\mathbf{d}-\mathbf{e}\in\ZZ^{Q_0}_{\geq0}$ is again a dimension vector. Then the quiver Grassmannian $Gr_\mathbf{e}(\tilde{M}_\mathbf{d})$ is non--empty if and only if $[\tilde{M}_\mathbf{e},\tilde{M}_\mathbf{d-e}]^1=0$.
\end{corollary}
\begin{proof}
Suppose that $[\tilde{M}_\mathbf{e},\tilde{M}_\mathbf{d-e}]^1=0$. 
Let us consider the representation $R:=\tilde{M}_\mathbf{e}\oplus \tilde{M}_\mathbf{d-e}$. We have $\tilde{M}_\mathbf{d}\Leq R$. It follows that $[\tilde{M}_{\mathbf{e}}, \tilde{M}_\mathbf{d}]^1\leq [\tilde{M}_{\mathbf{e}}, R]^1=0$ and hence  $[\tilde{M}_{\mathbf{e}}, \tilde{M}_\mathbf{d}]=[\tilde{M}_{\mathbf{e}}, R]=\langle\mathbf{e},\mathbf{d}\rangle$. Since by construction $\tilde{M}_\mathbf{e}$ embeds into R, by Bongartz's result, it embeds into $\tilde{M}_\mathbf{d}$ too, proving that $Gr_\mathbf{e}(\tilde{M}_\mathbf{d})$ is non--empty.

On the other hand, suppose that $Gr_\mathbf{e}(\tilde{M}_\mathbf{d})$ is non--empty. Then, by corollary~\ref{Cor:EmbeddingGeneric}, $\tilde{M}_\mathbf{e}$ embeds into $\tilde{M}_\mathbf{d}$. Let $Q$ be a quotient of  $\tilde{M}_\mathbf{d}$ by an embedding of $\tilde{M}_\mathbf{e}$. Then $[\tilde{M}_\mathbf{e}, \tilde{M}_\mathbf{d-e}]^1\leq [\tilde{M}_\mathbf{e},Q]^1\leq [\tilde{M}_\mathbf{d}, \tilde{M}_\mathbf{d}]^1=0$ as desired. 
\end{proof}

Corollary~\ref{Cor:Criterion NonEmpty} can be reformulated in terms of generic extensions  \cite{RGen}. 

\begin{corollary}
If $[\tilde{M}_\mathbf{e},\tilde{M}_\mathbf{d-e}]^1=0$,  the generic extension of $\tilde{M}_\mathbf{d-e}$ by $\tilde{M}_\mathbf{e}$ is $\tilde{M}_\mathbf{d}$.
\end{corollary}
\begin{proof}
In view of Corollary~\ref{Cor:Criterion NonEmpty}, if $[\tilde{M}_\mathbf{e},\tilde{M}_\mathbf{d-e}]^1=0$ then $Gr_\mathbf{e}(\tilde{M}_\mathbf{d})$ is non--empty. In this case,  there is a short exact sequence \eqref{Eq:GenericSequence} whose middle term is rigid, and hence its endomorphism ring has minimal dimension (among all the representations of dimension vector $\mathbf{d}$). In view of  \cite[Lemma~2.1]{RGen} the proof is complete.
\end{proof}

\begin{remark}
An interesting homological criterion for non--emptiness of a quiver Grassmannian associated with an arbitrary $Q$--representation  can be found in  \cite{MR}. 
\end{remark}

The next result collects properties of the quiver Grassmannians associated with rigid representations of a Dynkin quiver.

\begin{theorem}\label{Thm:RigQuivGras}
Let $\mathbf{e},\mathbf{d}\in\ZZ^{Q_0}_{\geq0}$ be dimension vectors such that $\mathbf{d}-\mathbf{e}\in\ZZ^{Q_0}_{\geq0}$. If $[\tilde{M}_\mathbf{e},\tilde{M}_\mathbf{d-e}]^1=0$ then $Gr_\mathbf{e}(\tilde{M}_\mathbf{d})$ is smooth and irreducible of dimension $\langle\mathbf{e},\mathbf{d}-\mathbf{e}\rangle$. 
\end{theorem}
\begin{proof}
Let $[(N,\iota)]\in Gr_\mathbf{e}(M)$. Since $M$ is rigid, $[N,M/\iota(N)]^1\leq [M,M]^1=0$ and hence the tangent space at $p=[(N,\iota)]$ has dimension  $[N,M/\iota(N)]=\langle\mathbf{e},\mathbf{d}-\mathbf{e}\rangle$, proving smoothness. Consider the stratification $Gr_\mathbf{e}(M)=\coprod_{[N]}\mathcal{S}_{[N]}$. We know from Corollary~\ref{Cor:EmbeddingGeneric} that the rigid representation $L:=\tilde{M}_{\mathbf{e}}$ of dimension vector $\mathbf{e}$ embeds into $M$. From the dimension formula for the strata we get
$$
\begin{array}{r}
\textrm{dim }\mathcal{S}_{[N]}=[N,M]-[N,N]\leq [N,M/N]=\langle\mathbf{e},\mathbf{d}-\mathbf{e}\rangle=[L,M/L]=\\\\
{}[L,M]-[L,L]=\textrm{dim }\mathcal{S}_{[L]}.
\end{array}
$$
If equality holds, then $[N,M]-[N,N]=[N,M/N]$ and hence $[N,N]^1\leq[N,M]^1=0$, proving that $N$ is rigid and hence isomorphic to $L$. We conclude that $Gr_\mathbf{e}(M)=\overline{\mathcal{S}_{[L]}}$ is irreducible. 
\end{proof}

The next result provides an application of Theorem~\ref{Thm:Bongartz} to quiver Grassmannians which will be used later. 

\begin{corollary}\label{Cor:QGAlmostSplit}
Let $\xymatrix@1{0\ar[r]&\tau M\ar^\iota[r]&E\ar^\pi[r]&M\ar[r]&0}$ be an almost split sequence. Then the quiver Grassmannian $\textrm{Gr}_{\mathbf{dim }M}(E)$ is empty and the quiver Grassmannian $\textrm{Gr}_{\mathbf{dim }M}(\tau M\oplus M)$ is a reduced point. In particular, 
\begin{equation}\label{Eq:DimenPoint}
\textrm{dim Gr}_{\mathbf{dim }M}(\tau M\oplus M)=1>\langle\textbf{dim }M,\textbf{dim }\tau M\rangle=-1.
\end{equation}
\end{corollary}
\begin{proof}
Since $E$ and $M$ are both rigid, if $\textrm{Gr}_{\mathbf{dim }M}(E)$ was non--empty, then, by Corollary~\ref{Cor:EmbeddingGeneric},  it would contain $M$, which is not the case since $[M,E]=0$. 

The quiver Grassmannian $\textrm{Gr}_{\mathbf{dim }M}(\tau M\oplus M)$ contains the canonical embedding of $M$ into $\tau M\oplus M$. Let us show that this is its only point. Let $[(N,j)]\in \textrm{Gr}_{\mathbf{dim }M}(\tau M\oplus M)$. Suppose that $N$ is not isomorphic to $M$. Then every map $f:N\rightarrow M$ factors through $\pi$ by the almost split property. In other words, the map $\Hom(N,\pi): \Hom_Q(N,E)\rightarrow \Hom_Q(N,M)$ induced by $\pi$ is surjective and its kernel is $\Hom_Q(N,\tau M)$. From this we see that $[N,\tau M\oplus M]=[N,E]$. By theorem~\ref{Thm:Bongartz} this yields an embedding of $N$ into $E$, contradicting the emptiness of $Gr_\mathbf{dim M}(E)$. Thus $N\simeq M$. Since $[M,\tau M]=0$, the only embedding of $M$ into $\tau M\oplus M$ is the canonical one, proving that $\textrm{Gr}_\mathbf{dim M}(\tau M\oplus M)$ is just a point. The tangent space at this point is isomorphic to $\Hom_Q(M,\tau M)$ which is zero dimensional, proving that $\textrm{Gr}_\mathbf{dim M}(\tau M\oplus M)$  is a reduced point. 
\end{proof}

In Corollary~\ref{Cor:QGAlmostSplit} above the specific quiver Grassmannian $Gr_{\mathbf{dim}\,M}(M\oplus\tau M)$ was considered. The next result collects properties of the remaining quiver Grassmannians associated with $M\oplus\tau M$. 

\begin{proposition}\label{Prop:QGMTauM}
Let $M$ be a non--projective indecomposable $Q$--representation, and let $\mathbf{e}\neq\mathbf{dim}\,M$  be a dimension vector such that $Gr_\mathbf{e}(M\oplus \tau M)$ is non--empty. Then $Gr_\mathbf{e}(M\oplus \tau M)$  is smooth of dimension $\langle\mathbf{e},\mathbf{d}-\mathbf{e}\rangle$ where $\mathbf{d}:=\mathbf{dim }(M\oplus \tau M)$. 
\end{proposition}
\begin{proof}
For simplicity of notation, we put $F:=M\oplus \tau M$. As above, we denote by $E$ the middle term of a (and hence any) almost split sequence ending in $M$. Let $p=[(N,j)]\in Gr_\mathbf{e}(F)$. If every morphism $f\in\textrm{Hom }(N,M)$ is not split--epi (i.e. $M$ is not a direct summand of N) then the almost split property implies that $[N,F]=[N,E]$ and $N$ embeds into $E$ (this can be also deduced by Bongartz's Theorem~\ref{Thm:Bongartz}). In particular $[N, F/j(N)]^1\leq [E,F]^1=0$ and hence $p$ is a smooth point of $Gr_\mathbf{e}(F)$. 

If there is a homomorphism $f\in\Hom_Q(N,M)$ which is epi--split, then $M$ embeds into $N$. The quotient $N/M$ is a sub--representation of $(M\oplus\tau M)/M\simeq \tau M$; in other words $N$ is the middle term of an exact sequence $0\rightarrow M\rightarrow N\rightarrow U\rightarrow 0$ for some $U\subseteq \tau M$. Since $[U,M]^1\leq [\tau M,M]^1=0$, we see that $N\simeq M\oplus U$. The embedding of $N$ into $F$ has the form
$$
\xymatrix@C=80pt{
j:N=M\oplus U\ar^{\left[\begin{array}{cc} 1_M&f\\0&g\end{array}\right]}[r]& F=M\oplus\tau M
}
$$ 
where $g:N\rightarrow \tau M$ is a monomorphism. The quotient $F/N$ is isomorphic to $\tau M/g(U)$. Then $[N,F/N]^1=[U\oplus M, \tau M/U]^1=[M,\tau M/U]^1$. If $U=\tau M$ then the point $p$ is clearly smooth. If $U$ is a non--zero proper sub--representation of $\tau M$, then the almost split property implies that $\tau M/g(U)$ is a quotient of E and hence $[M,\tau M/U]^1\leq [M,E]^1=0$ and  $p$  is smooth in this case. If $U$ is the zero representation, then $N=M$ and $\mathbf{e}$ would be equal to $\mathbf{dim}\,M$, against the hypothesis. We have shown that for any point $p\in Gr_\mathbf{e}(F)$ the tangent space at $p$ has dimension  equal to $\langle\mathbf{e},\mathbf{d-e}\rangle$ and hence $\textrm{Gr}_\mathbf{e}(M)$ is smooth of dimension $\textrm{dim }Gr_\mathbf{e}(F)\leq T_p(Gr_\mathbf{e}(F))=\langle\mathbf{e},\mathbf{d}-\mathbf{e}\rangle$. In particular, we see that if $\mathbf{e}\neq\mathbf{dim}\,M$ is such that $\textrm{Gr}_\mathbf{e}(M)$ is non--empty, then $\langle\mathbf{e},\mathbf{d-e}\rangle\geq0$. Moreover, in view of \eqref{Eq:IneqDimQuivGrass} we conclude that   $\textrm{dim }Gr_\mathbf{e}(F)=\langle\mathbf{e},\mathbf{d}-\mathbf{e}\rangle$ as desired. 
\end{proof}

\begin{remark}
Theorem~\ref{Thm:New} below will imply that $Gr_\mathbf{e}(M\oplus \tau M)$ is also irreducible. 
\end{remark}

The next result characterizes pairs $(\mathbf{e},\mathbf{d})$  of dimension vectors such that there exists a quiver Grassmannian $\textrm{Gr}_\mathbf{e}(M)$ associated with a point $M\in \textrm{Im }p_2\subseteq R_\mathbf{d}$ which is smooth of minimal dimension. 
\begin{proposition}\label{Rem:SmoothMinDim}
Let $\mathbf{e},\mathbf{d}\in\ZZ^{Q_0}_{\geq0}$ be dimension vectors such that $\mathbf{d}-\mathbf{e}\in\ZZ^{Q_0}_{\geq0}$. We consider the subset of $R_\mathbf{d}$ given by 
$$
\mathcal{S}_{\mathbf{e},\mathbf{d}}=\{M\in R_\mathbf{d}|\, \textrm{Gr}_\mathbf{e}(M)\textrm{ is non--empty and smooth of dimension }\langle\mathbf{e},\mathbf{d-e}\rangle\}.
$$ 
Then $\mathcal{S}_{\mathbf{e},\mathbf{d}}$ is non--empty if and only if $[\tilde{M}_\mathbf{e},\tilde{M}_\mathbf{d-e}]^1=0$. In this case, $\mathcal{S}_{\mathbf{e},\mathbf{d}}$ is open and dense in $R_\mathbf{d}$.
\end{proposition}
\begin{proof}
If $[\tilde{M}_\mathbf{e},\tilde{M}_\mathbf{d-e}]^1=0$ then $\tilde{M}_\mathbf{d}\in \mathcal{S}_{\mathbf{e},\mathbf{d}}$ which is hence non--empty. 
On the other hand, let $M\in \mathcal{S}_{\mathbf{e},\mathbf{d}}$. Then there is a point $[(N,j)]\in Gr_\mathbf{e}(M)$. We have $[N,M/j(N)]^1=0$ by assumption. It thus follows that  $[\tilde{M}_\mathbf{e}, \tilde{M}_\mathbf{d-e}]^1\leq [N,M/j(N)]^1=0$.

If $\mathcal{S}_{\mathbf{e},\mathbf{d}}$ is non--empty, then it contains the open orbit and hence it is dense in $R_\mathbf{d}$. Moreover, $\mathcal{S}_{\mathbf{e},\mathbf{d}}$ is (finite) union of $G_\mathbf{d}$--orbits (since the map $p_2:\textrm{Gr}_\mathbf{e}^Q(\mathbf{d})\rightarrow R_\mathbf{d}$ is $G_\mathbf{d}$--equivariant) and so is its complement. By upper--semicontinuity of the fiber dimension and of the dimension of the tangent space we see that the complement is closed. 
\end{proof}

\section{Main results}\label{Sec:Cluster}

This section contains the main results of the paper, already discussed in the introduction. In the whole section $Q$ denotes a Dynkin quiver and a representation of $Q$ is complex and finite dimensional. 

%
Recall that two projective varieties $X_1$ and $X_2$ are called \emph{deformation equivalent} if they are fibres of a proper smooth family over a connected base space. In this case $X_1$ and $X_2$ are diffeomorphic but the opposite is not true in general (there is a vast literature concerning this topic. For high dimension the reader could look at the classical papers \cite{LW1, LW2}, in dimension two there is a more recent paper by Manetti \cite{Ma}). In particular $X_1$ and $X_2$ share the same topological invariants (e.g. Poincar\'e polynomials, Euler characteristic..) and moreover they also have the same Hodge numbers (which is not the case for diffeomorphic varieties).  

\begin{theorem}\label{Thm:New}
Let $\mathbf{e}, \mathbf{d}\in\ZZ_{\geq 0}^{Q_0}$ such that $\mathbf{d-e}\in \ZZ_{\geq0}^{Q_0}$. Suppose that the set $\mathcal{S}_{\mathbf{e},\mathbf{d}}$ defined in Proposition~\ref{Rem:SmoothMinDim} is non--empty. Then for every $M_1,M_2\in \mathcal{S}_{\mathbf{e},\mathbf{d}}$ the quiver Grassmannians $Gr_\mathbf{e}(M_1)$ and $Gr_\mathbf{e}(M_2)$ are deformation equivalent. In particular, they are all diffeomorphic, irreducible, with the same Poincar\'e polynomial  and hence same Euler characteristic. Moreover they have the same Hodge numbers. 
\end{theorem}
\begin{proof}
Let us consider the universal quiver Grassmannian $Gr_\mathbf{e}^Q(\mathbf{d})\subset R_\mathbf{d}\times Gr_\mathbf{e}(\mathbf{d})$ and the map $p_\mathbf{d}:Gr_\mathbf{e}^Q(\mathbf{d})\rightarrow R_\mathbf{d}$ induced by the projection to $R_\mathbf{d}$. It was already observed that the map $p_\mathbf{d}$ is proper and $G_\mathbf{d}$--equivariant. Let us consider the restriction 
$$
p_\mathbf{d}|:p_\mathbf{d}^{-1}(\mathcal{S}_{\mathbf{e},\mathbf{d}})\rightarrow \mathcal{S}_{\mathbf{e},\mathbf{d}}.
$$
Since $\mathcal{S}_{\mathbf{e},\mathbf{d}}$ is non--empty by hypothesis, it is open and dense in $R_\mathbf{d}$, in particular it is smooth and irreducible. The counterimage $p^{-1}(\mathcal{S}_{\mathbf{e},\mathbf{d}})$ is smooth and irreducible being open and dense in the irreducible smooth variety $Gr_\mathbf{e}^Q(\mathbf{d})$. The restriction map $p_\mathbf{d}|$ is proper, since $p_\mathbf{d}$ is. By hypothesis, the fiber of $p_\mathbf{d}|$ over a point $N\in \mathcal{S}_{\mathbf{e},\mathbf{d}}$ is the quiver Grassmannian $Gr_\mathbf{e}(N)$ which is smooth and of dimension $\langle\mathbf{e}, \mathbf{d-e}\rangle$. Since the fibers have all the same dimension, $p_\mathbf{d}|$ is flat (see  \cite[Corollary of Theorem~23.1]{Mats}).  A proper flat morphism with smooth fibers is smooth \cite[Theorem~3', Ch.~III.10]{Mumf}. This shows that $Gr_\mathbf{e}(M_1)$ and $Gr_\mathbf{e}(M_2)$ are deformation equivalent.  By Ehresmann's trivialisation theorem (see e.g. \cite[Theorem~9.3]{Voisin}), $p_\mathbf{d}|$ is locally trivial, and hence (since Y is connected) its fibers are all diffeomorphic. In particular, all the fibers share the same topological invariants. They also have the same Hodge numbers: indeed Hodge numbers are upper semi--continuous and they sum up to the dimension of the cohomology spaces which are topological invariants. 
\end{proof}

\begin{remark}
The proof of Theorem~\ref{Thm:New} is inspired by \cite[Proof of Theorem~3.2]{CFR} where the flatness of a restriction morphism of $p_\mathbf{d}$ was used to deduce that degenerate flag varieties are flat degenerations of flag varieties. 
\end{remark}

\begin{theorem}\label{Thm:ARQG}
Let $\xymatrix@C=8pt{
0\ar[r]&\tau M\ar^\iota[r]&E\ar^\pi[r]&M\ar[r]&0
}$ be an almost split sequence in $Rep(Q)$. Then the quiver Grassmannians $Gr_\mathbf{e}(M\oplus \tau M)$ and $Gr_\mathbf{e}(E)$ are deformation equivalent if $\mathbf{e}\neq \mathbf{dim }M$. In particular they are diffeomorphic, $\chi(Gr_\mathbf{e}(M\oplus \tau M))=\chi(Gr_\mathbf{e}(E))$ and they have the same Poincar\'e polynomial and the same Hodge polynomial. 
\end{theorem}
\begin{proof}
In view of Proposition~\ref{Prop:QGMTauM}, Proposition~\ref{Rem:SmoothMinDim} and Theorem~\ref{Thm:New}, it remains to prove $Gr_\mathbf{e}(E)$ is non--empty if and only if $Gr_\mathbf{e}(M\oplus \tau M)$ is non--empty. Since $E$ is rigid, if $Gr_\mathbf{e}(E)$ is non--empty then every representation with the same dimension vector as $E$ admits a subrepresentation of dimension vector $\mathbf{e}$, since the map $p_\mathbf{d}$ is surjective in this case. In particular,  $Gr_\mathbf{e}(M\oplus \tau M)$ is non--empty. Viceversa, if $Gr_\mathbf{e}(M\oplus \tau M)$ is non--empty, then $M\oplus \tau M$ belongs to $\mathcal{S}_{\mathbf{e,d}}$ which is hence non--empty and contains $E$ in view of Proposition~\ref{Rem:SmoothMinDim}. 
 \end{proof}

\begin{remark}
The diffeomorphism between $Gr_\mathbf{e}(E)$ and $Gr_\mathbf{e}(M\oplus\tau M)$ does not preserve the sub-representation types, in general. For example, it is not true that the rigid representation $\tilde{M}_\mathbf{e}$ embeds into $M\oplus\tau M$, even if it does embed into E. For a counterexample consider the quiver $Q:1\rightarrow 2 \rightarrow 3$ and the almost split sequence
$0\rightarrow P_2\rightarrow P_1\oplus S_2\rightarrow I_2\rightarrow 0$. Then $Gr_{(1,1,1)}(P_1\oplus S_2)=\mathcal{S}_{[P_1]}$ and $Gr_{(1,1,1)}(P_2\oplus I_2)=\mathcal{S}_{[I_2\oplus S_3]}$. They are both (reduced) points but with different sub-representation types. 
\end{remark}

In type A, Theorem~\ref{Thm:ARQG} can be straightened by proving that the two quiver Grassmannians $Gr_\mathbf{e}(E)$ and $Gr_\mathbf{e}(M\oplus\tau M)$ are actually isomorphic. This follows from the explicit description of the almost split sequences given in \cite{BR} (since a type A quiver algebra is a string algebra) and induction. On the other hand in type D and hence E this is not the case: the following is a (counter-)example. 

\begin{example}
Let 
$$
Q:\xymatrix{
&2\ar[d]&\\
1\ar[r]&4&3\ar[l]
}
$$
be a quiver of type $D_4$ with subspace orientation. Consider the almost split sequence
$$
\xymatrix@C=20pt{
0\ar[r]& (1,1,1,2)\ar[r]&(1,1,0,1)\oplus (1,0,1,1)\oplus (0,1,1,1)\ar[r]&(1,1,1,1)\ar[r]&0
}
$$
where the indecomposables are described by their dimension vectors. Let $E:=(1,1,0,1)\oplus (1,0,1,1)\oplus (0,1,1,1)$  and $F:= (1,1,1,2)\oplus (1,1,1,1)$. They have the following presentations
$$
\begin{array}{c|c}
E:\xymatrix@C=40pt@R=50pt{
&\mathbf{C}^2\ar_{\left[\begin{array}{cc}1&0\\0&0\\0&1\end{array}\right]}[d]&\\
\mathbf{C}^2\ar_{\left[\begin{array}{cc}1&0\\1&1\\0&1\end{array}\right]}[r]&\mathbf{C}^3&
\mathbf{C}^2\ar^{\left[\begin{array}{cc}0&0\\1&0\\0&1\end{array}\right]}[l]}
&
F:\xymatrix@C=40pt@R=50pt{
&\mathbf{C}^2\ar_{\left[\begin{array}{cc}1&0\\0&0\\0&1\end{array}\right]}[d]&\\
\mathbf{C}^2\ar_{\left[\begin{array}{cc}1&0\\1&0\\0&1\end{array}\right]}[r]&\mathbf{C}^3&
\mathbf{C}^2\ar^{\left[\begin{array}{cc}0&0\\1&0\\0&1\end{array}\right]}[l]}
\end{array}
$$
Notice that the restriction of both E and F to the sub-quiver of $Q$ obtained by removing vertex 1 defines the same representation. This is a general fact that follows from Ringel's paper \cite{R} (in a previous version of this paper, this fact played an important r\^ole). Let us consider dimension vector $\mathbf{e}=(1,1,1,2)$. The quiver Grassmannian $Gr_{(1,1,1,2)}(E)$ is $(\PP^2)^\vee$ blown up in the three points $P_0=[1:0:0]$, $P_1=[1:-1:1]$ and $P_2=[0:1:0]$. On the other hand  $Gr_{(1,1,1,2)}(F)$ is $(\PP^2)^\vee$ blown up in the three points $Q_0=[1:0:0]$, $Q_1=[1:-1:0]$ and $Q_2=[0:1:0]$. Notice that the three points $P_0$, $P_1$ and $P_2$ are in generic position while the three points $Q_0$, $Q_1$ and $Q_2$ are collinear. It follows that they are not isomorphic (to see this one can notice that $Gr_{(1,1,1,2)}(E)$ is Fano, while  $Gr_{(1,1,1,2)}(F)$ is not). 
\end{example}

\subsection{Positivity}
In this section we prove that quiver Grassmannians which are smooth of minimal dimension have positive Euler characteristic. This is based on the following key result. 
\begin{theorem}\label{Thm:NoOddInd}
For every indecomposable representation $M$ of a Dynkin quiver $Q$, and every dimension vector $\mathbf{e}$, the quiver Grassmannian $Gr_\mathbf{e}(M)$ has zero odd cohomology. In particular, $\chi(Gr_\mathbf{e}(M))\geq0$. 
\end{theorem}
\begin{proof}
Since by assumption $M$ is indecomposable, its support is all contained in a connected subquiver of Q. In particular, we can assume that $Q$ itself is connected.  Let $\Gamma$ be its AR-quiver. Since $Q$ is connected, $\Gamma$ is connected. We consider a total ordering on the set $\{(i;k)|i\in Q_0,\, 0\leq k\leq k(i)\}$ of vertices of $\Gamma$ generated by the relation: $(i;k)\leq (j;\ell)$ if either $k<\ell$ or if $k=\ell$ then there is an arrow $i\rightarrow j$ in $Q^{op}$. We proceed by induction on such an ordering.

If $M=M(i;0)$ is projective, then the non--empty quiver Grassmannians associated with $M$ are points and hence the result holds. We hence proceed by induction, and we assume that $M=M(i;k)$ is a non--projective indecomposable module (hence $k>0$) and $\xymatrix@C=10pt{
0\ar[r]&\tau M\ar^\iota[r]&E\ar^\pi[r]&M\ar[r]&0
}$ is the almost split sequence ending in it. We prove that $Gr_\mathbf{e}(M)$ has no odd cohomology. Since $\tau M=M(i;k-1)$, we can assume by induction that $Gr_\mathbf{g}(\tau M)$ has no odd cohomology for every dimension vector $\mathbf{g}$. Let us show that the same holds for $Gr_\mathbf{e}(E)$. This is based on the following lemma. 
\begin{lem}\label{Lemma:Poincare}
Let $N_1$ and $N_2$ be two rigid $Q$--representations such that $N_1\oplus N_2$ is rigid. Let $\mathbf{e}$ be a dimension vector such that $Gr_\mathbf{e}(N_1\oplus N_2)$ is non--empty. Then the Poincar\'e polynomial of $Gr_\mathbf{e}(N_1\oplus N_2)$ is expressed in terms of the Poincar\'e polynomials of the quiver Grassmannians associated with $N_1$ and $N_2$ by the formula
\begin{align}\label{Eq:FormulaPoincareGeneral}
P_{Gr_\mathbf{e}(N_1\oplus N_2)}(q)&=&\sum_{\mathbf{f}+\mathbf{g}=\mathbf{e}}q^{2\langle\mathbf{f},\mathbf{dim}N_2-\mathbf{g}\rangle}P_{Gr_\mathbf{f}(N_1)}(q)P_{Gr_\mathbf{g}(N_2)}(q)\\\nonumber
&=&\sum_{\mathbf{f}+\mathbf{g}=\mathbf{e}}q^{2\langle\mathbf{g},\mathbf{dim }N_1-\mathbf{f}\rangle}P_{Gr_\mathbf{f}(N_1)}(q)P_{Gr_\mathbf{g}(N_2)}(q)
\end{align}
In particular, if $Gr_\mathbf{f}(N_1)$ and $Gr_\mathbf{g}(N_2)$ have no odd--cohomology for all $\mathbf{f}+\mathbf{g}=\mathbf{e}$, then the same holds for $Gr_\mathbf{e}(N_1\oplus N_2)$. 
\end{lem}
\begin{proof}
Recall the following fact (see \cite[Section~4]{BB} or \cite{I} or \cite[Sec.~1]{DLP}): let $X$ be a complex, projective and smooth variety on which the one--dimensional torus $T=\mathbf{C}^\ast$ acts algebraically. Let $X^T$ be the set of T--fixed points. The set $X^T$ is a smooth projective variety whose irreducible components we denote by $V_i$. Let $X_i=\{x\in X|\,\lim_{\lambda\rightarrow0}\lambda\cdot x\in V_i\}$ be the subset of points of $X$ which are attracted by points of $V_i$. The subsets $X_i$ of $X$ form an  \emph{$\alpha$--partition} (see \cite[Sec.~1]{DLP} or \cite[Section~3]{BB}) in the following sense:  they can be indexed $X_1,\cdots, X_n$  in such a way that $X_1\cup X_2\cup\cdots\cup X_i$ is closed in $X$ for every $i=1,\cdots, n$.   In \cite[Theorem~4.3]{BB} it is shown that the map $X_i\rightarrow V_i$ which sends $x\mapsto\lim_{\lambda\rightarrow 0}\lambda\cdot x$ is a locally trivial affine bundle (in the Zariski topology) whose fibers are complex affine spaces of dimension $p_i$. The integer $p_i$ is defined as follows: the action of $T$ on $X$ induces a linear action of $T$ on the tangent space $T_p(X)$ at the fixed points $p\in X^T$ and $p_i$ is the complex dimension of the subspace where the torus acts with positive weights (this dimension is locally constant and hence $p_i$ is well--defined for any irreducible component $V_i$ of $X^T$). 
Then the Poincar\'e polynomial of $X$ and the Poincar\'e polynomial of $V_i$ are related by the following formula (see \cite{BB74} for a proof over any algebraically closed field)
\begin{equation}\label{Eq:PoincareGeneral}
P_X(q)=\sum_{i=1}^n t^{2p_i}P_{V_i}(q).
\end{equation}
In particular, formula~\eqref{Eq:PoincareGeneral} shows that if $X^T$ has no odd cohomology then the same holds for $X$. 

Let us apply formula~\ref{Eq:PoincareGeneral} in our situation. Following Derksen, Weyman and Zelevinsky \cite[proof of Proposition~3.2]{DWZ2}, we let the 1--dimensional torus $T=\mathbf{C}^\ast$ act on $N_1\oplus N_2$ by $\lambda\cdot (n_1,n_2):=( n_1, \lambda n_2)$ for all $n_1\in N_1$ and $n_2\in N_2$. This defines an automorphism of the $Q$--representation $N_1\oplus N_2$ and hence it descends to an action of $T$ on the quiver Grassmannian $X:=Gr_\mathbf{e}(N_1\oplus N_2)$.   The space of T--fixed points 
is 
$$
X^T=\coprod_{\mathbf{f}+\mathbf{g}=\mathbf{e}}Gr_\mathbf{f}(N_1)\times Gr_\mathbf{g}(N_2).
$$
Given a T--fixed point $p=(L,j)$, the torus $T$ acts on the tangent space $T_p(X)\simeq\textrm{Hom }(L,N_1\oplus N_2/j(L))$ by $\lambda f(\ell):=\lambda\cdot f(\lambda^{-1}\cdot \ell)$. Since the point $p$ has the form $p=(L_1\oplus L_2, j)$, where $L_1\subseteq N_1$, $L_2\subseteq N_2$ and j is the diagonal embedding, the tangent space at $p$ is given by 
$$
T_p(X)\simeq \bigoplus_{i\in\{1,2\}}\Hom_Q(L_i, N_1/L_1)\oplus \Hom_Q(L_i, N_2/L_2).
$$ 
By definition, $T$ acts with weight zero on $\Hom_Q(L_1, N_1/L_1)\oplus \Hom_Q(L_2, N_2/L_2)$, with weight 1 on $\Hom_Q(L_1, N_2/L_2)$  and with weight (-1) on $\Hom_Q(L_2, N_1/L_1)$. 

Since $X$ is smooth and irreducible, and so are all the quiver Grassmannians $Gr_\mathbf{f}(N_1)$ and $Gr_\mathbf{g}(N_2)$, formula~\eqref{Eq:FormulaPoincareGeneral} is hence an immediate consequence of formula~\eqref{Eq:PoincareGeneral}, by taking into account  the classical  K\"unneth formula to write 
$$
P_{Gr_\mathbf{f}(N_1)\times Gr_\mathbf{g}(N_2)}(q)=P_{Gr_\mathbf{f}(N_1)}(q)P_{Gr_\mathbf{g}(N_2)}(q).
$$ 
The second equality in \eqref{Eq:FormulaPoincareGeneral} is obtained by Poincar\'e duality.
\end{proof}

In view of Lemma~\ref{Lemma:Poincare}, we see that $Gr_\mathbf{e}(E)$ has no odd cohomology. Indeed, we write $E=E(1)\oplus\cdots\oplus E(t)$ as a direct sum of its indecomposable direct summands  (it can be proved that $t\leq 3$, but this is not important). In view of \eqref{Eq:ARQuiverVertices}, each summand $E(j)$ has the form $E(j)=M(k_j;\ell)$ with $(k_j;\ell)<(i;k)$. By induction we can assume that the quiver Grassmannians associated with $E(j)$ have no odd cohomology. Lemma~\ref{Lemma:Poincare} implies that the same holds true for $Gr_\mathbf{e}(E)$.  

We are now ready to prove the statement for M. If $\mathbf{e}=\mathbf{dim }M$ then $Gr_\mathbf{dim M}(M)$ is a point and the result is clear.
For every $\mathbf{e}\neq \mathbf{dim }M$, by Theorem~\ref{Thm:ARQG}, the projective variety $X:=Gr_\mathbf{e}(M\oplus \tau M)$ is smooth and irreducible, and it is diffeomorphic to $Gr_\mathbf{e}(E)$. In particular, this holds if $\mathbf{e}$ is a dimension vector such that $Gr_\mathbf{e}(M)$ is non--empty. In this case, we have 
$$
\begin{array}{c}
P_{Gr_\mathbf{e}(E)}(q)=P_{Gr_{e}(M\oplus\tau M)}(q)=\\\\=q^{2\langle\mathbf{e},\mathbf{dim }\tau M\rangle}P_{Gr_{e}(M)}(q)+{\displaystyle\sum_{\mathbf{f}+\mathbf{g}=\mathbf{e},\, \mathbf{f}\neq\mathbf{e}}}q^{2\langle\mathbf{f},\mathbf{dim }\tau M-\mathbf{g}\rangle}P_{Gr_{f}(M)}(q)P_{Gr_{\mathbf{g}}(\tau M)}(q)
\end{array}
$$
By inductive hypothesis, the polynomials $P_{Gr_\mathbf{e}(E)}(q)$ and $P_{Gr_{\mathbf{g}}(\tau M)}(q)$ (for any $\mathbf{g}$) have no odd powers of q; thus the same holds for $P_{Gr_{f}(M)}(q)$ (for any $\mathbf{f}$ appearing in the right hand side), since possible odd powers of $Q$ would appear with the same sign, and hence cancellation could not occur. In particular, $P_{Gr_\mathbf{e}(M)}(q)$ has no odd powers of $q$, as desired.
\end{proof}

\begin{corollary}\label{Cor:NoOddExc}
Every quiver Grassmannian $Gr_\mathbf{e}(\tilde{M}_\mathbf{d})$ associated with the rigid module in $R_\mathbf{d}$ has zero odd cohomology. In particular, $\chi(Gr_\mathbf{e}(\tilde{M}_\mathbf{d}))\geq0$. 
\end{corollary}
\begin{proof}
Let $\tilde{M}_\mathbf{d}=E(1)\oplus\cdots\oplus E(s)$ be the decomposition of $\tilde{M}_\mathbf{d}$ into its indecomposable direct summands. Lemma~\ref{Lemma:Poincare} together with Theorem~\ref{Thm:NoOddInd} implies the result. 
\end{proof}

\begin{corollary}
Let $\mathbf{e}$ and $\mathbf{d}$ be dimension vectors such that $\mathbf{d-e}$ is again a dimension vector. If $[\tilde{M}_\mathbf{e},\tilde{M}_{\mathbf{d-e}}]^1=0$ then every smooth quiver Grassmannian $Gr_\mathbf{e}(M)$ of dimension $\langle\mathbf{e},\mathbf{d-e}\rangle$ associated with $M\in R_\mathbf{d}$ has no odd cohomology. In particular $\chi(Gr_\mathbf{e}(M))\geq0$.
\end{corollary}
\begin{proof}
In view of Theorem~\ref{Thm:New},  $Gr_\mathbf{e}(M)$ is diffeomorphic to $Gr_\mathbf{e}(\tilde{M}_\mathbf{d})$ which has the required property in view of Corollary~\ref{Cor:NoOddExc}.
\end{proof}

\begin{remark}
The fact that the Euler characteristic of every quiver Grassmannian of Dynkin type is non--negative was proved by Caldero and Keller \cite[Theorem~3]{CK1} using Hall algebras and Lusztig's canonical bases. 

The fact that quiver Grassmannians associated with rigid representations of an arbitrary acyclic quiver have no odd cohomology, was proved by F.~Qin in \cite[Theorem~3.2.6]{Qin} as a consequence of its formula for the quantum F--polynomials of quantum cluster monomials. A geometric proof of this fact was obtained by Nakajima \cite[Theorem~A.1]{Naka}.  
\end{remark}

\subsection{Homology}
In this section we analyze the homology groups of the quiver Grassmannians (of Dynkin type) which are smooth and of minimal dimension. 
Recall from \cite[Sec.~1.7]{DLP} that  an algebraic variety $X$ is said to have property (\textbf{S})  if the following two properties are satisfied:
\begin{itemize}
\item[(\textbf{S}1)] $H_i(X)$ is zero if i is odd and it has no torsion if i is even; 
\item[(\textbf{S}2)] the cycle map $\varphi_i: A_i(X)\rightarrow H_{2i}(X)$ is an isomorphism for all i.  
\end{itemize}
(Here $H_i(X)$ denotes the Borel--Moore $i$--th homology group and $A_i(X)$ is the group generated by k--dimensional irreducible subvarieties modulo rational equivalences (see \cite[Sec.~1.3]{Fu}).) 
\begin{theorem}\label{Thm:(S)}
Let $Q$ be a Dynkin quiver and let $M$ be an indecomposable $Q$--representation. Let $\mathbf{e}\in\mathbf{Z}^{Q_0}_{\geq0}$ be a dimension vector such that the quiver Grassmannian $Gr_\mathbf{e}(M)$ is non--empty. Then $Gr_\mathbf{e}(M)$ has property (\textbf{S}1). 
\end{theorem}
\begin{proof}
Since by assumption $M$ is indecomposable, its support is all contained in a connected subquiver of Q. In particular, we can assume that $Q$ itself is connected.  Let $\Gamma$ be its AR-quiver. Since $Q$ is connected, $\Gamma$ is connected and it is acyclic. We choose a total ordering on the set $\{(i;k)|i\in Q_0,\, 0\leq k\leq k(i)\}$ of vertices of $\Gamma$ so that $(i;k)\leq (j;\ell)$ if  $k<\ell$ and $(i;k)\leq (j;k)$ whenever there is an arrow $i\rightarrow j$ in $Q^{op}$. We proceed by induction on such an ordering. 

If  $M=M(i;0)$ is projective indecomposable, then the non empty quiver Grassmannians associated with $M$ are points. In particular the claim holds in this case. 

We hence assume that $M=M(i;k)$ is not projective (i.e. $k>0$), and let 
$$
0\rightarrow \tau M\rightarrow E\rightarrow M\rightarrow 0
$$
be the almost split sequence ending in M. If $\mathbf{e}=\mathbf{dim}\, M$, then the quiver Grassmannian $Gr_{\mathbf{dim}\,M}(M)$ is a reduced point, and the claim holds. We hence fix  $\mathbf{e}\neq\mathbf{dim}\,M$ and such that $Gr_\mathbf{e}(M)$ is non--empty. Since $Gr_\mathbf{0}(M)$ is a point, by induction we can assume that every non--empty quiver Grassmannian $Gr_\mathbf{f}(M)$ has property (\textbf{S}1), for every $\mathbf{f}<\mathbf{e}$ (here $\mathbf{f}<\mathbf{e}$ means that $\mathbf{e-f}\in\ZZ_{\geq0}^{Q_0}$ and $\mathbf{e}\neq \mathbf{f}$). Since $\tau M=M(i;k-1)$, the inductive hypothesis guarantees that $Gr_\mathbf{g}(\tau M)$ has property (\textbf{S}1) for every $\mathbf{g}$. Let us show that by induction we can also assume that $Gr_\mathbf{e}(E)$ has property (\textbf{S}1). 
We decompose $E=E(1)\oplus\cdots\oplus E(t)$ as direct sum of its indecomposable direct summands. We let the torus $T=\mathbf{C}^\ast$ act on $Gr_\mathbf{e}(E)$ by $\lambda\cdot (x_1,x_2,\cdots, x_t):=(\lambda x_1, x_2, \cdots, x_t)$ and we consider the induced $\alpha$--partition (since E is rigid, $Gr_\mathbf{e}(E)$ is smooth and irreducible):
$$
Gr_\mathbf{e}(E)=\coprod_{\mathbf{f}}Gr_\mathbf{e}(E)^{\mathbf{f}}
$$
where every piece is the total space of an affine bundle 
$$
\xymatrix{
Gr_\mathbf{e}(E)^{\mathbf{f}}\ar@{->>}[r]&Gr_{\mathbf{f}}(E(1))\times Gr_{\mathbf{e-f}}(E(2)\oplus\cdots\oplus E(t)).
}
$$
By \eqref{Eq:ARQuiverVertices}, $E(1),\cdots, E(t)$ correspond to vertices of $\Gamma$ which are smaller than $(i;k)$, and hence by induction we can assume that both $Gr_{\mathbf{f}}(E(1))$ and $Gr_{\mathbf{e-f}}(E(2)\oplus\cdots\oplus E(t))$ have property (\textbf{S}1). In particular, their product $Gr_{\mathbf{f}_1}(E(1))\times Gr_{\mathbf{f}_2}(E(2)\oplus\cdots\oplus E(t))$ has property (\textbf{S}1). We will use freely the following well--known fact: if $\xymatrix@C=15pt{E\ar@{->>}[r]&X}$ is an affine bundle (locally trivial in the Zarisky topology) and $X$ has property (\textbf{S}) then E has property (\textbf{S}) (see \cite[Sec.~1.8]{DLP}).  The stratum $Gr_\mathbf{e}(E)^{\mathbf{f}}$ has property (\textbf{S}1) and hence, by \cite[Sec.~1.8]{DLP}, we conclude that the whole variety $Gr_\mathbf{e}(E)$ has property (\textbf{S}1). 

We can now prove that $Gr_\mathbf{e}(M)$ has property (\textbf{S}1). In view of Theorem~\ref{Thm:ARQG}, $Gr_\mathbf{e}(E)$ is diffeomorphic to $Gr_\mathbf{e}(M\oplus\tau M)$. In particular, $Gr_\mathbf{e}(M\oplus\tau M)$ has property (\textbf{S}1) and it is smooth and irreducible.  We let the torus $T=\mathbf{C}^\ast$ act on $Gr_\mathbf{e}(M\oplus\tau M)$ by 
\begin{equation}\label{Eq.TorusAct(S)}
\lambda\cdot (m,n):=(m, \lambda  n)
\end{equation} 
for every $\lambda\in\mathbf{C}^\ast$, $m\in M$ and $n\in \tau M$. The variety 
$Gr_\mathbf{e}(M\oplus\tau M)$ has a corresponding $\alpha$--partition
$$
Gr_\mathbf{e}(M\oplus\tau M)=\coprod_{\mathbf{f}} Gr_\mathbf{e}(M\oplus\tau M)^\mathbf{f}
$$
where $Gr_\mathbf{e}(M\oplus\tau M)^\mathbf{f}$ is the total space of an affine bundle
$$
\xymatrix{
Gr_\mathbf{e}(M\oplus\tau M)^\mathbf{f}\ar@{->>}[r]&Gr_\mathbf{f}(M)\times Gr_\mathbf{e-f}(\tau M)
}
$$
of rank $\langle\mathbf{f},\mathbf{dim}\,\tau\,M-\mathbf{e+f}\rangle$ (see Lemma~\ref{Lemma:Poincare}). For simplicity of notation, we put $X:=Gr_\mathbf{e}(M\oplus\tau M)$. With our choice \eqref{Eq.TorusAct(S)} of the torus action, we see that the stratum $X^\mathbf{0}\simeq Gr_\mathbf{e}(\tau M)$ is closed in $X$ while the stratum $X^\mathbf{e}$ is open in $X$. We put $U:=X^\mathbf{e}$ and we notice that $U$ is the total space of an affine bundle over $Gr_\mathbf{e}(M)$.  The closed complement of $U$ is 
$$
Y:=\coprod_{\mathbf{f}<\mathbf{e}}X^\mathbf{f}.
$$
For every $\mathbf{f}<\mathbf{e}$ the stratum $X^\mathbf{f}$ is the total space of an affine  bundle on $Gr_\mathbf{f}(M)\times Gr_\mathbf{e-f}(\tau M)$. Since by inductive hypothesis $Gr_\mathbf{f}(M)\times Gr_\mathbf{e-f}(\tau M)$ has property (\textbf{S}1),  each stratum $X^\mathbf{f}$ has property (\textbf{S}1) and  hence $Y$ itself has property (\textbf{S}1). For every i, we have an exact sequence
$$
\xymatrix@C=15pt{
H_{2i+1}(X)\ar[r]&H_{2i+1}(U)\ar^h[r]&H_{2i}(Y)\ar^f[r]&H_{2i}(X)\ar^g[r]&H_{2i}(U)\ar[r]&H_{2i-1}(Y)
}
$$
Both $X$ and $Y$ have property (\textbf{S}1). Let us show that the same holds for $U$. Since $H_{2i+1}(X)=0$ the homomorphism $h$ is injective and hence $H_{2i+1}(U)$ is torsion--free, being a subgroup of the torsion--free group $H_{2i}(Y)$. Since there is an affine bundle $\xymatrix{U\ar@{->>}[r]&Gr_\mathbf{e}(M)}$ and by Theorem~\ref{Thm:NoOddInd} the base space $Gr_\mathbf{e}(M)$ has no odd cohomology, we see that also $U$ has no odd cohomology. By Poincar\'e duality, this means that $H_{2i+1}(U)\otimes \mathbf{Q}$ is zero. But since $H_{2i+1}(U)$ is torsion--free, we conclude that $H_{2i+1}(U)=0$. It follows that $H_{2i+1}(Gr_\mathbf{e}(M))=0$.  It remains to check that $H_{2i}(Gr_\mathbf{e}(M))$ is torsion--free. We cannot see this from the sequence above. Instead, we change torus action on $Gr_\mathbf{e}(M\oplus \tau M)$: for every $\lambda\in\mathbf{C}^\ast$, $m\in M$ and $n\in \tau M$ we define
\begin{equation}\label{Eq:TorusAction2(S)}
\lambda\ast(m,n):=(\lambda m, n).
\end{equation}
The corresponding $\alpha$--partition is 
$$
Gr_\mathbf{e}(M\oplus\tau M)=\coprod_{\mathbf{f}} Gr_\mathbf{e}(M\oplus\tau M)^\mathbf{f}
$$
where $Gr_\mathbf{e}(M\oplus\tau M)^\mathbf{f}$ is the total space of an affine  bundle
$$
\xymatrix{
Gr_\mathbf{e}(M\oplus\tau M)^\mathbf{f}\ar@{->>}[r]&Gr_\mathbf{f}(M)\times Gr_\mathbf{e-f}(\tau M)
}
$$
of rank $\langle\mathbf{e-f},\mathbf{dim}\,M-\mathbf{f}\rangle$ (see Lemma~\ref{Lemma:Poincare}). This $\alpha$--partition is dual with respect to the one above: attracting sets for one action are repulsive sets for the other. In particular, the closed stratum  is now $Y':=Gr_\mathbf{e}(M\oplus\tau M)^\mathbf{e}\simeq Gr_\mathbf{e}(M)$.  The complement is the open subset
$$
U':=X-Y'=\coprod_{\mathbf{f}<\mathbf{e}} Gr_\mathbf{e}(M\oplus\tau M)^\mathbf{f}.
$$
For all i, we have an exact sequence: 
$$
\xymatrix{
H_{2i+1}(U')\ar[r]&H_{2i}(Gr_\mathbf{e}(M))\ar[r]&H_{2i}(X).
}
$$ 
By induction we can assume that $H_{2i+1}(U')=0$. It follows that $H_{2i}(Gr_\mathbf{e}(M))$ is  a subgroup of $H_{2i}(X)$ which is torsion--free by assumption. We conclude that $H_{2i}(Gr_\mathbf{e}(M))$ is torsion--free and hence $Gr_\mathbf{e}(M)$ has property (\textbf{S}1). 
\end{proof}
\begin{remark}
I conjecture that $Gr_\mathbf{e}(M)$ has property (\textbf{S}2) as well. An evidence for this conjecture is given by the fact that  the Hodge number $h^{p,q}(Gr_\mathbf{e}(\tilde{M}_\mathbf{d}))=0$ if $p\neq q$ (this can be proved with a similar argument as in Corollary~\ref{Cor:NoOddExc}). Indeed the image of the cycle map $\varphi_p$ is always in the $(p,p)$--component.

What is missing to prove this conjecture, is the fact that such a property is not preserved under deformation. This is because the chow groups $A_i(X)$ are not invariant under deformation, while the homology groups are. 

In type A, since  $Gr_\mathbf{e}(E)$ and $Gr_\mathbf{e}(M\oplus\tau M)$ are isomorphic, then this applies.  Actually, for quivers of type A equioriented, a much stronger result is valid: every quiver Grassmannian admits a cellular decomposition (i.e. an $\alpha$--partition into affine spaces)
\cite[Theorem~12]{CFFGR} and a space admitting a cellular decomposition has property (\textbf{S}) \cite[Sec. 1.10]{DLP}. Caldero and Keller \cite{CK1} conjecture that this is true for any Dynkin quiver. 
\end{remark}

\begin{corollary}
Every smooth quiver Grassmannian (of Dynkin type) of minimal dimension has property (\textbf{S}1). 
\end{corollary}
\begin{proof}
Let $\tilde{M}_\mathbf{d}$ be the rigid $Q$--represention in $R_\mathbf{d}$ and let $\mathbf{e}$ be a dimension vector such that $Gr_\mathbf{e}(\tilde{M}_\mathbf{d})$ is non--empty. In view of Theorem~\ref{Thm:New} it is enough to prove that $Gr_\mathbf{e}(\tilde{M}_\mathbf{d})$ has property (\textbf{S}1).  We proceed by induction on the number of indecomposable direct summands of $\tilde{M}_\mathbf{d}$. If $\tilde{M}_\mathbf{d}$ is indecomposable, then the claim is just Theorem~\ref{Thm:(S)} above. We write $\tilde{M}_\mathbf{d}=E(1)\oplus\cdots\oplus E(t)$ as a direct sum of its indecomposable direct summands ($t\geq 2$), we let the torus $T=\mathbf{C}^\ast$ act on $Gr_\mathbf{e}(\tilde{M}_\mathbf{d})$ by $\lambda\cdot (x_1,x_2,\cdots, x_t):=(\lambda x_1, x_2, \cdots, x_t)$ for any $x_i\in E(i)$. We consider the induced $\alpha$--partition
$$
Gr_\mathbf{e}(\tilde{M}_\mathbf{d})=\coprod_{\mathbf{f}}Gr_\mathbf{e}(\tilde{M}_\mathbf{d})^{\mathbf{f}}
$$
where every piece is the total space of an affine bundle 
$$
\xymatrix{
Gr_\mathbf{e}(\tilde{M}_\mathbf{d})^{\mathbf{f}}\ar@{->>}[r]&Gr_{\mathbf{f}}(E(1))\times Gr_{\mathbf{e-f}}(E(2)\oplus\cdots\oplus E(t))
}
$$
(the fact that this is an affine  bundle follows from \cite{BB} by using that $Gr_\mathbf{e}(\tilde{M}_\mathbf{d})$ is a smooth and irreducible projective variety). By induction we assume that both $Gr_{\mathbf{f}}(E(1))$ and $Gr_{\mathbf{e-f}}(E(2)\oplus\cdots\oplus E(t))$ have property (\textbf{S}1). In view of \cite[Lemma~1.9]{DLP},  $Gr_\mathbf{e}(\tilde{M}_\mathbf{d})^{\mathbf{f}}
$ has property (\textbf{S}1). In view of \cite[Lemma~1.8]{DLP} the whole variety $Gr_\mathbf{e}(\tilde{M}_\mathbf{d})$ has property (\textbf{S}1). 
\end{proof}

\section{Applications to cluster algebras}\label{Sec:GVect}
In this section we provide a new proof of the formula of Caldero and Chapoton \cite{CC}. This formula associates to each indecomposable $Q$--representation $M$ a Laurent polynomial CC($M$) in $n$--variables, so that the indecomposables correspond to the non--initial cluster variables of the cluster algebras associated with $Q$. The formula can be formulated in full generality, but we restrict to the coefficient--free setting, for simplicity. The CC--formula consists of two ingredients: the $\mathbf{g}$--vector and the F--polynomial of M. 

\subsection{\textbf{g}--vector of a $Q$--representation}
Let $A=KQ$ be the path algebra of a Dynkin quiver $Q$, over the field of complex numbers K. The Grothendieck group of $A$, denoted with $K_0(Q)$,  is, by definition, the free abelian group generated by isoclasses $[M]$ of  A-modules factored out by the subgroup generated by  $[L]+[N]-[M]$ whenever there is a short exact sequence $0\rightarrow L\rightarrow M\rightarrow N\rightarrow 0$ in $A$--mod. The set of elements $[M]$ such that $M$ is an indecomposable $A$--module is denoted with $\textrm{ind }K_0(Q)$. By definition, if $M\simeq \bigoplus M(k)^{m_k}$ is the decomposition of an $A$--module $M$ as a direct sum of its indecomposable direct summands, then $[M]=\sum a_k[M_k]$. Thanks to the Jordan-H\"older property of $A$--mod, the set $\mathcal{S}:=\{[S_i]|\, i\in Q_0\}$ of isoclasses of simple $A$--modules, form a basis of $K_0(Q)$ and the coordinate vector of an element $[M]\in K_0(Q)$ in this basis is nothing but its dimension vector $\mathbf{dim}\,M$ if $M\in A-\textrm{mod}$. 
We fix the standard basis $\{\alpha_i\}_{i\in Q_0}$ of $\ZZ^{Q_0}$ and the map $[S_i]\mapsto \alpha_i$ identifies $K_0(Q)$ with $\ZZ^{Q_0}$ so that $\mathbf{dim}\, M=\sum_{i\in Q_0}d_i\alpha_i$. In view of its homological interpretation \eqref{Eq:EulForm}, the Euler form descends to a bilinear form on $K_0(Q)$ and hence on $\ZZ^{Q_0}$. Let $H$ be the matrix representing this form in the basis $\mathcal{S}$: $\langle\mathbf{e},\mathbf{d}\rangle=\mathbf{e}^t\,\textrm{H}\,\mathbf{d}$. The matrix $H=(h_{ij})_{i,j\in Q_0}$ is given by
$$
h_{ij}=
\left\{
\begin{array}{cc}
1&\textrm{ if }i=j;\\
-1&\textrm{ if there is an arrow }i\rightarrow j;\\
0&\textrm{otherwise.}
\end{array}
\right.
$$
Since $Q$ is acyclic, the Euler form satisfies the following identities:
\begin{equation}\label{Eq:ProjSimpleEuler}
\langle [P_i], [S_j]\rangle=(\mathbf{dim}\,P_i)^t\,H\,\mathbf{dim}\,S_j=[P_i,S_j]=\delta_{ij}
\end{equation}
From \eqref{Eq:ProjSimpleEuler} we see that $H$ is invertible (over $\ZZ$) and its inverse has the i--th \emph{row} equal to $\mathbf{dim}\,P_i$. It is customary to define $C=(c_{ij})_{i,j\in Q_0}$ to be the matrix whose j--th \emph{column} is $\mathbf{dim}\,P_j$ which means
\begin{equation}\label{Eq:C}
c_{ij}=\#\{\textrm{paths from j to i in Q}\}
\end{equation}
In particular, the $k$--th row of $C$ is $\mathbf{dim}\, I_k$ so that the transpose matrix, denoted with $C^t$, has $\mathbf{dim}\,I_k$ for its $k$--th column. From the above discussion we have:
\begin{equation}\label{Eq:HC}
C^tH=\mathbf{1}.
\end{equation}
The matrices $C$ and $H=(C^t)^{-1}$ are respectively called the \emph{Cartan matrix} and the \emph{Euler matrix} of $Q$ (see e.g. \cite{ASS} or \cite{Ralf}). Notice that $C$ does not coincide with the Cartan matrix of the Lie algebra associated with $Q$.  

Since $C$ is invertible (over $\ZZ$), both the set $\mathcal{P}=\{[P_i]\in K_0(Q)|\,i\in Q_0\}$ of isoclasses of projective indecomposables and  the set $\mathcal{I}=\{[I_j]|\, j\in Q_0\}$ of isoclasses of injectives indecomposables form a basis of $K_0(Q)$ (equivalently, this can be deduced by the fact that every $M\in A$--mod admits an essentially  unique minimal projective and minimal injective resolution). We put $\omega_j:=[I_j]$ so that  the group $K_0(Q)$ is identified with the lattice $\bigoplus_{j\in Q_0}\ZZ \omega_j$ with respect to the basis $\mathcal{I}$. This choice is motivated by the standard convention in Lie theory that fundamental weights are $\{\omega_i\}$ while simple roots are $\{\alpha_i\}$ (see Remark~\ref{Rem:Weights}).

\begin{definition}\label{Def:GVect}
Let $X\in K_0(A)$. The \emph{$\mathbf{g}$--vector} or \emph{index} of X, denoted with $\mathbf{g}_X$, is the coordinate vector of $-X$  in the basis $\mathcal{I}$. 
\end{definition}
The name $\mathbf{g}$--vector comes from the Fomin--Zelevinsky theory of cluster algebras (see \cite{FZIV}).
\begin{remark}\label{Rem:Weights}
The reason why we put a minus sign in the definition of $\mathbf{g}_M$ is that we want $\mathbf{g}_{I_j}=-\omega_j$. As explained in \cite{YZ}, $\mathbf{g}$--vectors should be thought as weights, while dimension vectors as roots. 
\end{remark}
One can be more explicit, and  in view of \eqref{Eq:MinInjRes}, $[M]=[I_0^M]-[I_1^M]$   for any $M\in A$--mod, we get the explicit formula for the i--th coordinate of $\mathbf{g}_M$:
$$
(\mathbf{g}_M)_i=-[S_i,M]+[S_i,M]^1=- \langle S_i,M\rangle.
$$

\begin{lem}\label{Cor:GAlmostSplit}
Let $\xymatrix@C=10pt{
0\ar[r]&A\ar^\iota[r]&B\ar^\pi[r]&C\ar[r]&0
}$ be a short exact sequence in Rep($Q$). Then $\mathbf{g}_B=\mathbf{g}_A+\mathbf{g}_C$. In particular, for any $M$ and N
\begin{equation}\label{Eq:GSum}
\mathbf{g}_{M\oplus N}=\mathbf{g}_M+\mathbf{g}_N.
\end{equation}
\end{lem}
\begin{proof}
The i--th component of $\mathbf{g}_B$ equals $-\langle S_i, B\rangle$ which equals $-\langle S_i, A\rangle-\langle S_i, C\rangle$ (to see this, apply the covariant functor $\textrm{Hom }(S_i,-)$ to the given exact sequence). 
\end{proof}

By definition, $\mathbf{dim}\, M$ and $-\mathbf{g}_M$ are the coordinate vectors of $[M]$ respectively in the basis $\mathcal{S}$ and in the basis $\mathcal{I}$ of $K_0(Q)$. Since the columns of $C^t$ are the coordinate vectors of the elements of $\mathcal{I}$ in the basis $\mathcal{S}$, and $H$ is its inverse, we get
\begin{equation}\label{Eq:GDim}
\mathbf{g}_M=-\textrm{H}\,\mathbf{dim}\, M.
\end{equation}

Dually to the notion of an index, or $\mathbf{g}$--vector, is the notion of the \emph{coindex}. 
\begin{definition}
The \emph{coindex} of [M], denoted with $\mathbf{g}^M$, is the coordinate vector of 
$-[M]\in K_0(A)$ in the basis $\mathcal{P}$. 
\end{definition}
The following result establishes a relation between index, coindex  and $\tau$.
\begin{lem}
For every indecomposable non--projective representation $M$ we have 
\begin{equation}\label{Eq:GTau}
\mathbf{g}_{\tau M}=-\mathbf{g}^M=H^t\,\mathbf{dim}\, M.
\end{equation}
\end{lem}
\begin{proof}
In view of \eqref{Eq:ProjReso} we get
$
(\mathbf{g}^M)_i=-\langle  M, S_i\rangle=-d_i+\sum_{j\rightarrow i}d_j=-\sum h_{ji}d_j 
$
and hence
$\mathbf{g}^M=-H^t\,\mathbf{dim}\,M$.  The first equality follows from \eqref{Eq:MinInjResTau}, using \eqref{Eq:PropNu}. 
\end{proof}

\begin{corollary}
For any indecomposable non--projective module M, we have 
\begin{align}\label{Eq:Cox}
\mathbf{dim }\tau M&=&-(H^{-1})(H^t)\mathbf{dim }M\\
\label{Eq:CoxG}
\mathbf{g}_{\tau M}&=&-(C^{-1})(C^t)\,\mathbf{g}_M.
\end{align}
\end{corollary}
The matrix $\Phi:=-(H^{-1})(H^t)$ is called the \emph{Coxeter matrix} of $Q$. The reason for this name is the following: if $M$ is non projective, then both $\mathbf{dim }M$ and $\mathbf{dim }\tau M$ are roots for the underlying  Dynkin diagram of $Q$. Let $s_1,\cdots s_n$ be the simple reflections generating the corresponding Weyl group W. The possible orientations of the underlying Dynkin diagrams are in bijection with the Coxeter elements of W, which are the elements $c=s_{i_1}s_{i_2}\cdots s_{i_n}$ for all possible $i_k\neq i_{\ell}\in[1,n]$: given an orientation $Q$
$$
c_Q^{-1}:=c_{Q'}\circ \left(\prod_{i\in Q_0:\\ i\textrm{ is a sink}}s_i\right)
$$
where $Q'$ is obtained from $Q$ by removing all the sinks. For example 
$$
\xymatrix@R=3pt{&&&&4&&\\Q:&1\ar[r]&2\ar[r]&3\ar[ru]\ar[rd]&&&c_Q^{-1}=s_1\circ s_2\circ s_3\circ s_4\circ s_5\\&&&&5&&
}
$$
where the Weyl group acts on the roots as functions. Then it can be proved that 
\begin{equation}\label{Eq:TauCox}
\textbf{dim }\tau M=c_Q^{-1}(\textbf{dim }M)
\end{equation}
and also
\begin{equation}\label{Eq:TauCoxG}
\textbf{g}_{\tau M}=c_Q^{-1}(\textbf{g}_M)
\end{equation}

\begin{remark}
Equation~\ref{Eq:TauCox} provides a convenient way to compute $\textbf{dim }\tau M$: indeed for any vertex i, 
$$
\textbf{dim }(s_i(\mathbf{d}))_j=\left\{\begin{array}{cc}d_j&\textrm{if }j\neq i\\ -d_i+\sum_{(k-i) \in Q_1}d_k&\textrm{if }j=i\end{array}\right.
$$
\end{remark}

We are now ready to compare  $\mathbf{g}_{\tau M}$ and $\mathbf{g}_M$, for any indecomposable non--projective $A$--module $M$. We need the following definition 
\begin{definition}\label{Def:ExchMatrix} 
The \emph{exchange matrix} of  $Q$ is the  matrix 
$B:=H-H^t$.
\end{definition}
From the definition, it follows that the $ij$--entry $b_{ij}$ of the matrix $B$ is given by
$b_{ij}=\#\{j\rightarrow i\in Q_1\}-\#\{i \rightarrow j\in Q_1\}$ (for $i,j\in Q_0$)\footnote{In the original paper of Caldero and Chapoton \cite{CC}, the authors made dual choices:  they defined the  $\mathbf{g}$--vector of $M$ as what I define to be the coindex of $M$. In order to make theorem~\ref{Thm:BTauG} working, this definition lead them to work with the opposite exchange matrix. Nowadays, after the work of Derksen-Weyman-Zelevinsky \cite{DWZ1, DWZ2} it is customary to associate to $Q$ the exchange matrix as in definition~\ref{Def:ExchMatrix}, and this motivates my choice of the definition of $\mathbf{g}$.}. Once again, the name of $B$ comes from the Fomin--Zelevinsky theory of cluster algebras.

\begin{theorem}\label{Thm:BTauG}
For any indecomposable non--projective $A$--module M, we have
\begin{equation}\label{Eq:G-vector}
\mathbf{g}_M+\mathbf{g}_{\tau M}+\textrm{B }\mathbf{dim}\,M=0
\end{equation}
\end{theorem}
\begin{proof}
In view of \eqref{Eq:GDim} and \eqref{Eq:GTau}, we have: 
$$
\textrm{B}\,\mathbf{dim}\,M=\textrm{H}\,\mathbf{dim}\,M-\textrm{H}^t\,\mathbf{dim}\,M=-\mathbf{g}_{M}-\mathbf{g}_{\tau M}.
$$
\end{proof}

\subsection{F--polynomial of a $Q$--representation}
Given a $Q$--representation $N$, its F--polynomial is the generating function of the Euler characteristic of the quiver Grassmannians associated with N: 
$$
F_N(y_1,\cdots, y_n):=\sum_{\mathbf{e}\in \ZZ_{\geq 0}^{Q_0}}\chi (Gr_\mathbf{e}(N))\mathbf{y}^\mathbf{e}
$$
where $\mathbf{y}^\mathbf{e}:=\prod_{i\in Q_0} y_i^{e_i}$. 
Let us discuss some properties of $F$--polynomials. First of all, since isomorphic representations have isomorphic quiver Grassmannian, the F--polynomial $F_M$ is constant along the isoclass of M.  The definition is formally extended to the elements $\{-[P_i]|\, i\in Q_0\}$ by declaring
\begin{equation}
F_{-[P_i]}(\mathbf{y}):=1,\qquad (\forall i\in Q_0).
\end{equation}
\begin{proposition}\label{Prop:DWZDirectSum}
For any $M,N\in \textrm{A--mod}$,  $F_{M\oplus N}=F_MF_N$. 
\end{proposition}
\begin{proof}
This proof is due to Derksen--Weyman--Zelevinsky \cite{DWZ2}. The 1--dimensional torus $T=\mathbf{C}^\ast$ acts on $M\oplus N$ by $\lambda\cdot (m,n):=(m,\lambda n)$. This defines an automorphism of $M\oplus N$ and hence it descends to an action of $T$ on the quiver Grassmannian $Gr_\mathbf{e}(M\oplus N)$. The T--fixed points are direct sums of subrepresentations of $M$ and of N: $Gr_\mathbf{e}(M\oplus N)^T=\prod_{\mathbf{f}+\mathbf{h}=\mathbf{e}}Gr_\mathbf{f}(M)\times Gr_\mathbf{h}(N)$. In particular, $\chi(Gr_\mathbf{e}(M\oplus N))=\chi(Gr_\mathbf{e}(M\oplus N)^T)=\sum_{\mathbf{f}+\mathbf{h}=\mathbf{e}}\chi(Gr_\mathbf{f}(M)) \chi(Gr_\mathbf{h}(N))$ which proves the proposition.
\end{proof}
\begin{remark}
Proposition~\ref{Prop:DWZDirectSum} should be compared with Lemma~\ref{Lemma:Poincare}. While in Proposition~\ref{Prop:DWZDirectSum} there are no assumptions on  $M$, $N$ and $M\oplus N$, this is not the case for Lemma~\ref{Lemma:Poincare}: the reason is that the Euler characteristic of a projective varieties on which a torus acts equal the Euler characteristic of the space of T--fixed points, with no assumptions on X. But in Lemma~\ref{Lemma:Poincare} it was needed that the quiver Grassmannians associated with $M$, $N$ and $M\oplus N$ were all smooth. 
\end{remark}
\begin{theorem}
For any almost split sequence $\xymatrix@C=8pt{
0\ar[r]&\tau M\ar^\iota[r]&E\ar^\pi[r]&M\ar[r]&0
}$ the following formula holds:
\begin{equation}\label{Eq:CCFormulaF}
F_{M}(\mathbf{y})F_{\tau M}(\mathbf{y})=F_E(\mathbf{y})+\mathbf{y}^{\text{dim }M}.
\end{equation}
\end{theorem}
\begin{proof}
In view of Proposition~\ref{Prop:DWZDirectSum}, $F_{M\oplus\tau M}=F_MF_{\tau M}$. 
The rest follows by Theorem~\ref{Thm:ARQG}: 
\begin{align*}
F_{M\oplus \tau M}(\mathbf{y})&=&\sum_{\mathbf{e}}\chi(Gr_\mathbf{e}(M\oplus \tau M))\mathbf{y}^\mathbf{e}\\
&=&\sum_{\mathbf{e}\neq\mathbf{dim}\,M}\chi(Gr_\mathbf{e}(E))\mathbf{y}^\mathbf{e}+\mathbf{y}^{\mathbf{dim}\,M}\\
&=&F_E(\mathbf{y})+\mathbf{y}^{\mathbf{dim}\,M}.
\end{align*}
\end{proof}

\subsection{CC-formula}\label{Sec:CCFormula}
We begin the section by recalling the CC--formula, formulated in terms of $\mathbf{g}$--vectors and $F$--polynomials. 
\begin{definition}
Given a $Q$--representation $M$,  the Laurent polynomial $CC(M)\in \ZZ[x_1^{\pm1},\cdots, x_n^{\pm1}]$ (where $n=|Q_0|$) is defined as follows
$$
CC(M)=F_M(\mathbf{x}^{B^1},\cdots, \mathbf{x}^{B^n})\,\mathbf{x}^{\mathbf{g}_M}
$$
where $B^1,\cdots, B^n$ are the n columns of the matrix B (see definition~\ref{Def:ExchMatrix}).  More explicitly
\begin{equation}\label{Eq:CC}
CC(M)=\sum_{\mathbf{e}\in\ZZ_{\geq0}^n}\chi(Gr_\mathbf{e}(M))\,\mathbf{x}^{B\mathbf{e}+\mathbf{g}_M}.
\end{equation}
\end{definition}
\begin{remark}
In view of Corollary~\ref{Cor:Criterion NonEmpty} formula \eqref{Eq:CC} can be refined as follows: if $M\simeq\tilde{M}_\mathbf{d}$ is rigid (in particular if $M$ is indecomposable), then 
$$
CC(M)=\sum_{\mathbf{e}:\, [\tilde{M}_\mathbf{e},\tilde{M}_\mathbf{d-e}]^1=0}\chi(Gr_\mathbf{e}(M))\,\mathbf{x}^{B\mathbf{e}+\mathbf{g}_M}.
$$
\end{remark}
\begin{proposition}\label{Prop:CCPlus}
Given two $Q$--representations $M$ and $N$, we have $$CC(M\oplus N)=CC(M)CC(N).$$
\end{proposition}
\begin{proof}
Since $F_{M\oplus N}=F_MF_N$ (by Proposition~\ref{Prop:DWZDirectSum}) and $\mathbf{g}_{M\oplus N}=\mathbf{g}_{M}+\mathbf{g}_{N}$ (by Lemma~\ref{Cor:GAlmostSplit}), the result follows. 
\end{proof}
\begin{theorem}
\begin{enumerate}
\item For every $k\in Q_0$
\begin{equation}\label{CCFormulaInj}
CC(I_k)x_k=(\prod_{k\rightarrow i}x_i) (\prod_{j\rightarrow k} CC(I_j))+1.
\end{equation}
\item For any almost split sequence $\xymatrix@C=8pt{
0\ar[r]&\tau M\ar^\iota[r]&E\ar^\pi[r]&M\ar[r]&0
}$ starting in a non--injective module $\tau M$:
\begin{equation}\label{CCFormula}
CC(\tau M)CC(M)=CC(E)+1.
\end{equation}
\end{enumerate}
In particular, the set $\{CC(M)|\, M\in ind(Q)\}\cup\{x_i|\,i\in Q_0\}$ is the set of all cluster variables of the (coefficient--free) cluster algebra associated with $Q$ and \eqref{CCFormulaInj}--\eqref{CCFormula} are all the primitive exchange relations. 
\end{theorem}
\begin{proof}
\begin{enumerate}
\item The simple module $S_k$ is the socle of $I_k$ and hence $S_k$ is a sub-representation of any non--zero sub-representation of $I_k$. We denote by $R$ the corresponding quotient. It is immediate to see that  
$
R\simeq\bigoplus_{j\rightarrow k}I_j.
$
We have
\begin{align*}
CC(I_k)&=&\mathbf{x}^{\mathbf{g}_{I_k}}\left(\sum_\mathbf{e}\chi(Gr_\mathbf{e}(I_k))\mathbf{x}^{B\mathbf{e}}\right)\\
&=&x_k^{-1}\left(\sum_{\mathbf{e}=\mathbf{e}'+\mathbf{dim}\,S_k}\chi(Gr_\mathbf{e'}(R))\mathbf{x}^{B\mathbf{e'}+B\mathbf{dim}\,S_k}+1\right)\\
&=&x_k^{-1}\left(\sum_{\mathbf{e'}}\chi(Gr_\mathbf{e'}(R))\mathbf{x}^{B\mathbf{e'}}\prod_{j\rightarrow k}x_j^{-1}\prod_{k\rightarrow i}x_i+1\right)\\
&=&x_k^{-1}\left(\sum_{\mathbf{e'}}\chi(Gr_\mathbf{e'}(R))\mathbf{x}^{B\mathbf{e'}+\mathbf{g}_R}\prod_{k\rightarrow i}x_i+1\right)\\
&=&x_k^{-1}\left(CC(R)\prod_{k\rightarrow i}x_i +1\right).
\end{align*}
and \eqref{CCFormulaInj} follows from Proposition~\ref{Prop:CCPlus} (in the  third equality the definition of the matrix B has been used). 
\item In view of Proposition~\ref{Prop:CCPlus}, $CC(M\oplus \tau M)=CC(M)CC(\tau M)$. We have:
\begin{align*}
CC(M\oplus \tau M)&=&F_{M\oplus\tau M}(\mathbf{x}^{B^1},\cdots, \mathbf{x}^{B^n})\mathbf{x}^{\mathbf{g}_{M\oplus\tau M}}\\
&=&(F_E(\mathbf{x}^{B^1},\cdots, \mathbf{x}^{B^n})+\mathbf{x}^{B\mathbf{dim}\,M})\mathbf{x}^{\mathbf{g}_{M\oplus\tau M}}\\
&=&F_E(\mathbf{x}^{B^1},\cdots, \mathbf{x}^{B^n})\mathbf{x}^{\mathbf{g}_E}+\mathbf{x}^{B\mathbf{dim}\,M+\mathbf{g}_M+\mathbf{g}_{\tau M}}\\
&=&CC(E)+1.
\end{align*}
The second equality follows from \eqref{Eq:CCFormulaF}; the third one follows from Lemma~\ref{Cor:GAlmostSplit}; the last equality follows from Theorem~\ref{Thm:BTauG}.
\end{enumerate}
In view of \cite[Theorem~1.5]{YZ} and \eqref{Eq:TauCoxG}, equations \eqref{CCFormulaInj} and \eqref{CCFormula} are precisely the primitive exchange relations of the (coefficient--free) cluster algebra associated with $Q$ and hence the last statement follows by induction. 
\end{proof}

\section*{Acknowledgments}
I thank Corrado De Concini for discussions on torus actions on projective varieties and property (\textbf{S}). I thank Marco Manetti, Gabriele Mondello, Giulio Codogni, Ernesto Mistretta and Kieran O'Grady for discussions concerning Ehresmann's trivialisation theorem. I thank Salvatore Stella for discussions on $\mathbf{g}$--vectors. I thank Claus-Michael Ringel for discussions concerning radiation modules (used in a previous version of the paper). I am indebted with Markus Reineke and Evgeny Feigin, without whom this paper would have not been possible. I thank Francesco Esposito for never--ending generic discussions. I specially thank DFG-SPP-1388 and in particular Peter Littelmann for having supported my project ``Categorification of positivity in cluster algebras''. I am grateful to an anonymous referee for his meticulous reading of the manuscript and many deep remarks and suggestions.

This work was supported by the national FIRB project ``Perspectives in Lie Theory'' RBFR12RA9W.

\end{document}